\documentclass{amsart}
\usepackage[utf8]{inputenc}
\usepackage{times,amsfonts,mathabx,amssymb,amsrefs,mathrsfs,tikz, amsmath, enumitem}
\usetikzlibrary{decorations,decorations.pathmorphing}
\usepackage{todonotes}
\usepackage[colorlinks=true,linkcolor=blue,citecolor=magenta]{hyperref}

\usepackage{bbm}

\newcommand{\C}{\mathbb{C}}

\newcommand{\Z}{{\mathbb Z}}
\newcommand{\R}{{\mathbb R}}

\newcommand{\Tcal}{{\mathcal T}}
\newcommand{\X}{\mathcal{X}}

\renewcommand{\Pr}{{\mathbb P}}

\newcommand{\arr}{\longrightarrow}
\newcommand{\G}{\Gamma}
\newcommand{\alb}{\mathsf{X}}
\newcommand{\xs}{\alb^*}
\newcommand{\xo}{\alb^\omega}
\newcommand{\xmo}{\alb^{-\omega}}
\newcommand{\nuke}{\mathcal{N}}
\newcommand{\img}[1]{\mathop{\mathrm{IMG}\left(#1\right)}}
\newcommand{\lims}[1][G]{\mathcal{J}_{#1}}
\newcommand{\bim}{\mathfrak{B}}
\newcommand{\limg}{\mathcal{X}_G}
\newcommand{\si}{\mathsf{s}}

\newcommand{\cdim}{\mathop{\mathrm{ARdim}}}

\DeclareMathOperator{\diam}{diam}
\DeclareMathOperator{\supp}{supp}

\DeclareMathOperator{\Capc}{Cap}

\newcommand{\gbf}{\mathbf{g}}
\newcommand{\sbf}{\mathbf{s}}

\newcommand{\Tbf}{\mathbf{T}}

\newtheorem{thm}{Theorem}[section]
\newtheorem{prop}[thm]{Proposition}

\newtheorem{cor}[thm]{Corollary}
\newtheorem{lem}[thm]{Lemma}
\theoremstyle{definition}
\newtheorem{defn}[thm]{Definition}

\newtheorem{rem}[thm]{Remark}

\newtheorem{notation}[thm]{Notation}

\newtheorem*{thm*}{Theorem}
\newtheorem*{cor*}{Corollary}
\newcommand{\thistheoremname}{}

\newtheorem*{genericthm*}{\thistheoremname}
\newenvironment{namedthm*}[1]
  {\renewcommand{\thistheoremname}{#1}%
   \begin{genericthm*}}
  {\end{genericthm*}}
\newtheorem*{genericcor*}{\thistheoremname}
\newenvironment{namedcor*}[1]
  {\renewcommand{\thistheoremname}{#1}%
   \begin{genericthm*}}
  {\end{genericthm*}}
\title{Liouville property for groups and conformal dimension}
\date{\today}

\author{Nicol\'as Matte Bon}
\address{}
\author{Volodymyr Nekrashevych}
\author{Tianyi Zheng}

\thanks{The work of NMB was supported by the LABEX MILYON (ANR-10-LABX-0070) of Université de Lyon, within the program "France 2030" (ANR-11-IDEX-0007) operated by the French National Research Agency (ANR). VN was supported by NSF grant DMS2204379. TZ was supported by a Sloan research fellowship and NSF grant DMS2348143.}
\begin{document}

\maketitle

\begin{abstract}

Conformal dimension is a fundamental invariant of metric spaces, particularly suited to the study of self-similar spaces, such as spaces with  expanding self-coverings (e.g., Julia sets of complex rational functions). The dynamics of these systems are encoded by the associated iterated monodromy groups, which are examples of contracting self-similar groups. Their amenability is a well-known open question. We show that if $G$ is an iterated monodromy group, and if the (Ahlfors-regular) conformal dimension of the underlying space is strictly less than 2, then every symmetric random walk with finite second moment on $G$ has the Liouville property. As a corollary, every such group is amenable. This criterion applies to all examples of contracting groups previously known to be amenable, and to many new ones. In particular, it implies that for every sub-hyperbolic complex rational function $f$ whose Julia set is not the whole sphere, the iterated monodromy group of $f$ is amenable.

\end{abstract}

\section{Introduction}

Conformal dimension was introduced in the late 1980s by P.~Pansu in the study of quasi-isometries of hyperbolic spaces and related constructions (see, for example,~\cite {pansu}). The conformal dimension of a metric space is the infimum of the Hausdorff dimensions of spaces quasi-symmetric to it. Since quasi-isometries of hyperbolic spaces correspond to quasi-symmetries of their boundaries, the conformal dimension is a natural invariant of hyperbolic groups and their boundaries.
Since its definition, the conformal (and related Ahlfors-regular conformal) dimension has become an important invariant in geometric group theory and dynamical systems.  See~\cite{mackay_tyson} for a survey of its properties and applications.

Ahlfors-regular conformal dimension is especially useful and natural for self-similar metric spaces. Boundaries of hyperbolic groups are examples of such self-similar spaces. Another natural class of examples is metric spaces with an expanding (branched) covering map $f\colon \mathcal{J}\arr\mathcal{J}$ (for example Julia sets of hyperbolic complex rational functions).
Their quasi-conformal geometry is the subject of the works~\cite{haisinskypilgrim,haispilgr:subhyperbolic,haispilgr:inequiv} by P.~Ha\"{\i}ssinsky and K.~Pilgrim.

Expanding self-coverings are encoded by the associated \emph{iterated monodromy groups}, introduced in the early 2000s, see~\cite{nek:book,nek:dyngroups}. They are examples of \emph{contracting self-similar groups}. The self-covering is reconstructed from the iterated monodromy group as their \emph{limit dynamical system}. The potential importance of the conformal dimension to the study of contracting self-similar groups was observed in~\cite{haispilgr:subhyperbolic}.

The amenability of iterated monodromy groups has been a well-known open question since their introduction. In fact, amenability of the iterated monodromy group of the complex polynomial $z^2-1$ (also known as the \emph{Basilica group}) was asked even before the iterated monodromy groups were defined (see~\cite{zukgrigorchuk:3st}).

Amenability of the Basilica group was proved by L.~Bartholdi and B.~Vir\'ag in~\cite{barthvirag}. Their proof was later generalized to include iterated monodromy groups of all post-critically finite complex polynomials in~\cite{bkn:amenability} (and more generally, all groups generated by \emph{bounded} automata). Their methods were further extended in~\cite{amirangelvirag:linear} to all groups generated by automata of \emph{linear activity} (in the sense of S.~Sidki~\cite{sid:cycl}).

Amenability was proved in these papers by constructing a self-similar random walk with step distribution $\mu$  on a special family of ``mother groups'', and using the self-similarity to prove that the asymptotic entropy (equivalently, speed) of the $\mu$-random walk is zero (see also~\cite{Kai-munchaussen}). By the results  of  V.~Kaimanovich and A.~Vershik \cite{Kai-Ver} and Y.~Derriennic \cite{Der} this implies the triviality of the Poisson boundary of the corresponding random walk, also known as the \emph{Liouville property} that all bounded $\mu$-harmonic functions are constant, which in turn implies  amenability of the group.

A different approach to prove the amenability of  groups is based on the notion of extensively amenable actions, first introduced by K. Juschenko and N. Monod \cite{juschenkomonod}. This method was used in  \cite{YNS} to show amenability of a class of iterated monodromy groups. This approach allows one to prove the amenability of iterated monodromy groups of complex rational functions with  ``crochet'' Julia sets, and of all contracting self-simlar groups generated by automata of \emph{polynomial activity}, see~\cite{nekpilgrimthurston}. The framework of extensive amenability is not based on the study of the Liouville property for random walks, but random walks are implicitly present in it, as one of the main conditions used to prove amenability is the recurrence of the induced random walk on an orbit of some action of the group. 

We show in our paper how conformal geometry can be used to prove the Liouville property for symmetric random walks under suitable moment conditions on many contracting groups (equivalently, iterated monodromy groups of expanding covering maps).

Namely, we prove the following.

\begin{thm} \label{t-intro-Liouville}
    Let $(G, \bim)$ be a finitely generated contracting self-similar group, and suppose that its limit space has  Ahlfors-regular conformal dimension $\alpha$, with $\alpha<2$. Then for every symmetric measure $\mu$ on $G$ having finite $\beta$-moment for some $\beta>\alpha$, the $\mu$-random walk is Liouville. In particular, any such group is amenable.
\end{thm}

This implies the following (the relevant terminology on rational functions is recalled in \S\ref{ss:img}). 

\begin{cor}
    Let $f$ be a sub-hyperbolic complex rational function. If the Julia set of $f$ is not the whole sphere (in particular, if $f$ is hyperbolic), then its iterated monodromy group is amenable.
\end{cor}

Our results cover many new examples of amenable self-similar groups. Namely, all previously known contracting amenable groups had limit spaces of conformal dimension equal to 1 (this is true for any contracting group which can be generated by an automaton of polynomial activity, see Theorem \ref{th:polactivity}). In this case,  Theorem \ref{t-intro-Liouville} implies that any random walk on them generated by symmetric measures of finite $(1+\varepsilon)$-moment is Liouville, encompassing various previous results for some more restricted classes of groups and measures.
Beyond this case, Theorem \ref{t-intro-Liouville} can be applied to contracting groups whose amenability was an open question until now, in particular to groups whose limit space is homeomorphic to the Sierpi\'nski carpet. As shown in \cite{haispilgr:inequiv}, there exist hyperbolic rational maps with Sierpi\'nski carpet Julia sets whose conformal dimensions are arbitrarily close to 2 (the corresponding groups must therefore be generated by automata of exponential activity). 

\subsection*{Critical exponent and moment conditions}
The moment condition on the measure in Theorem \ref{t-intro-Liouville} is tightly related to another numeric invariant of a contracting self-similar group: its \emph{critical exponent}. For every $p\in (0, \infty]$ we define the \emph{$\ell^p$-contraction coefficient} $\eta_p$ of a finitely generated contracting group $(G, \bim)$. It is the exponential rate of decay of the $\ell^p$-norm of the vector of lengths of the sections $g|_v$ for the vertices $v$ of  the $n$th level of the tree. The condition for a self-similar group to be contracting is equivalent to the condition $\eta_\infty<1$. The function $p\mapsto \eta_p$ is continuous and non-increasing. The critical exponent $p_c(G, \bim)$ is the infimum of the values $p$ for which $\eta_p<1$. It is a measure of the complexity of the self-similar group. For example, for any $\varepsilon>0$, the action on the tree of every element $g\in G$ of length $n$ can be explicitly described with $O(n^{p_c+\epsilon})$ bits of information (more precisely, the \emph{portrait} of $g$ has size $O(n^{p_c+\epsilon})$, see Proposition \ref{prop:portraisize}). In particular, $p_c<1$ implies sub-exponential volume growth. In our setting, such complexity bounds can be used to control the contribution to entropy from long jumps of the random walk. 
In analogy to critical constant for recurrence introduced in \cite{erschler-critical}, for a finitely generated group $G$ equipped with a word metric, we define the (symmetric) {\it critical constant for the Liouville property}, denoted by $\rm{Cr}_{Liouv}(G)$, to be the infimum of $p$, such that all symmetric random walks with finite $p$-moment are Liouville. For a contracting group $(G, \bim)$, we denote by $\cdim(G, \bim)$ the Ahlfors-regular conformal dimension of its limit space. Between these critical quantities and conformal dimension, we show the following inequality.  
\begin{thm}\label{th-intro-ineq}
Let $(G, \bim)$ be a finitely generated contracting self-similar group such that $\cdim(G, \bim)<2$. Then
\[
\rm{Cr}_{Liouv}(G)\le p_c(G, \bim)\le \cdim(G, \bim).
\]
\end{thm}
Then Theorem \ref{t-intro-Liouville} above is a consequence of this inequality. The second inequality is not sharp in general, for example for the Grigorchuk group we have $\cdim(G, \bim)=1$ and $\rm{Cr}_{Liouv}(G)=p_c(G, \bim)\simeq 0.7674$; see the discussion after Theorem \ref{th:critexp}.   The first inequality is strict for some examples where the group $G$ is virtually nilpotent: for instance if $(G, \bim)$ is the \emph{adding machine} (so that $G\simeq \Z$) then $\rm{Cr}_{Liouv}(G)=0$ and $p_c(G, \bim)=1$. We do not know whether this is essentially the only source of examples.

\medskip
\subsection*{Further background and comparison with previous work}

 Let us give some further insight on previous results on amenability and the Liouville property for self-similar groups, and compare them with Theorem \ref{t-intro-Liouville}. Recall that a contracting group $G$ is generated by a finite state-automata over a finite alphabet $\alb$, and has a natural action on the associated rooted tree $\alb^*$ whose vertex set consists of finite words. For every integer $n$, the action of an element $g\in G$ on $\alb^*$ is determined by two pieces of information: the permutation induced on the $n$-th level of the tree, and associated collection of {sections} $(g|_v)_{v\in \alb^n}$, which describe the action of $g$ on the subtrees rooted at level $n$. 
 
In order to show vanishing of the asymptotic entropy of the random walk associated to a measure $\mu$, one is immediately led to study  the sections $\gbf_t|_v$  of a sample $\gbf_t=\mathbf{h}_t\cdots \mathbf{h}_1$ of the $\mu$-random walk. More precisely, it is enough to show that the rate of growth of the sum of the entropies (or alternatively, the word lengths) of  all sections $\gbf_t|_v$ at level $n$ is bounded by some constant which tends to 0   as $n$ tends to infinity (Proposition \ref{p-Munchaussen}).  The sequence of sections $(\mathbf{g}_t|_v)_{t\ge 1}$ satisfies the cocycle rule $\mathbf{g}_{t+1}|_v=(\mathbf{h}_{t+1}|_{\gbf_t(v)})\cdot ( \gbf_t|_v)$, so in order to study it, one is also led to study the behaviour of the process $(\gbf_t(v))_{t\ge 1}$, which performs a random walk on the finite Schreier graph associated with the $n$-th level of the tree.

These observations (in various related formulations) have already been the starting point of previous results establishing the Liouville property for random walk on groups generated by automata of bounded or linear activity \cite{barthvirag, bkn:amenability,amirangelvirag:linear, AAMV}. This means that for every element $s$ in a generating set of the group, the number of vertices $v\in \alb^n$ such that the section $s|_v$ is non-trivial is bounded \cite{barthvirag, bkn:amenability, AAMV}, or grows linearly with $n$ \cite{amirangelvirag:linear}. This feature was used in \cite{barthvirag, bkn:amenability, amirangelvirag:linear} to construct special classes of self-similar random walks, for which the distribution of the process $\gbf_t|_v$ can be analysed by explicit recursive computations based on the form of the automaton. Self-similarity is a very restrictive condition in this case, and it is not clear how to construct such self-similar random walks in general (it seems that a finitely supported self-similar step distribution does not exist in most cases), so in \cite{bkn:amenability,amirangelvirag:linear} the corresponding groups have to be embedded into special ``mother groups'' for which such a self-similar random walk can be constructed. 

For groups generated by bounded automata,  this self-similarity restriction on the random walk was removed  in \cite{AAMV}, by showing that every symmetric finitely supported measure generates a Liouville random walk. This was done by analysing explicitly the Schreier graphs of the action of the group on finite levels of the tree. In fact in this case, to understand the behaviour of the sections $\gbf_t|_v$, it is enough to understand the visits of $(\gbf_t(v))$ to some uniformly bounded subsets of ``singular points'' on the graphs (namely the vertices $w$ such that the section $h|_w$ is non-trivial for some $h$ in the support of the measure). The length of sections $\gbf_t|_v$ is controlled by the number of \emph{traverses} of the random walk between pairs of singular points lying far away on the Schreier graph, which in turn can be bounded by estimating the $\ell^2$-\emph{capacity} (or \emph{effective conductance}). 

The existence of such special finite sets of ``singular points'' for each group element can be interpreted in terms of germs of the action on the boundary of the tree $\alb^\omega$, and is also crucial for the application of the method of extensive amenability \cite{juschenkomonod, YNS, amirangelvirag:quadratic, nekpilgrimthurston}. The idea of this method is to identify a large amenable subgroup $H$ of the group whose action has no singular points, and then to exploit the recurrence of the Schreier graphs to control the expected number of singular points along the random walk, which allows us to deduce amenability from the amenability of $H$.

For a typical group as in Theorem \ref{t-intro-Liouville}, the approaches outlined above meet two related difficulties. First, in many cases (e.g., for the iterated mondromy group of a sufficiently general hyperbolic rational function), any explicit automaton generating the group will necessarily have an exponential activity and  might be quite complicated; considerations based on the form of such an automaton would quickly become intractable. Second, there is no natural ``finite set of singular points'' associated to each group element  (the set of singularities will rather correspond to a Cantor subset of the boundary of the tree). For a contracting self-similar group, the finiteness of the singular points seems to be related to the existence of countable separating subsets in the limit space. This is a restrictive condition limited to examples with limit space of conformal dimension 1, see  \cite{nekpilgrimthurston}. 

\subsection*{Outline of the proof} Our proof of Theorem \ref{th-intro-ineq} also proceeds by reducing the problem to the analysis of random walk on Schreier graphs, but this is achieved by establishing direct relationships between the properties of a contracting self-similar group, and those of its limit space. 
 We crucially use the fact that the finite Schreier graphs can be seen as approximations of the limit space. There are three main ingredients in our proof.

The first ingredient, purely geometric, is Theorem \ref{th:traverses}, which estimates  the total length of $n$-th level sections of a word $g=s_t\cdots s_1$  in terms of the  paths that the word describes on the corresponding Schreier graph. It says that in order for the total length of sections to be large, these paths must traverse many times between certain regions which approximate some disjoint parts of the limit space (which roughly speaking correspond to different \emph{tiles}). The precise identification of these regions requires the framework of \emph{contracting models} developed in~\cite{nek:models,nek:dimension}.

The second ingredient is Theorem \ref{th:energies}, where the assumption that the conformal dimension of the limit space is less than 2 is used to show that the $\ell^2$-capacity between such disjoint regions must tend to 0 as $n\to \infty$. This is reminiscent of the results of M. Carrasco \cite{carrasco} and J. Kigami \cite{kigamibook}, which relate the conformal dimension of a metric space to potential theory on graphs that approximate the space (although our proof of Theorem \ref{th:energies} is self-contained). 
These two ingredients put us in a position to apply a traverse-counting argument to show the vanishing of random walk asymptotic entropy, at least for symmetric finitely supported measures.

The third ingredient is the study of critical exponent (as described above), which allows us to go beyond the case of finitely supported measures.

\subsection*{On return probability} Let us also mention that it is known that for symmetric random walks with finite second moment, sufficiently slow decay of return probabilities, namely $\mu^{(2n)}(id)$ decaying slower than $\exp(-n^{1/2})$, implies that the $\mu$-random walk is Liouville, see \cite{PZ20}. However this result is not applicable in the study of contracting self-similar groups: first, it is not clear for what subclass of groups one can expect such a slow decay, secondly, even in the case of bounded automata groups, it is challenging to prove lower estimates for return probabilities. Indeed, to the best of our knowledge, for all contracting groups whose amenability has been proven so far and those covered by Theorem \ref{t-intro-Liouville}, the best known lower bound for $\mu^{(2n)}(id)$ comes from entropy, that is, $\mu^{(2n)}(id)\ge \exp(-H(X_{2n}))$. It is an open problem to find methods to establish sharper lower estimates for return probabilities, even in examples such as the Basilica group, and a tightly related problem is to find contracting groups that are amenable but not Liouville for symmetric simple random walks (see \S~\ref{s-non-Liouville} for an example of non-Liouville contracting group).

\subsection*{Organisation of the article} 
Section 2  is devoted to the theory of self-similar contracting groups and their limit spaces. After recalling various preliminaries from this theory, we prove Theorem \ref{th:traverses}, which is at the core of the first ingredient described above.

In Section 3 we describe the quasi-conformal structure of the limit space of a contracting self-similar group and introduce its Ahlfors-regular conformal dimension. The limit space can be seen as the boundary of a natural Gromov-hyperbolic graph (the \emph{self-similarity complex} of the group), and therefore can be endowed with the associated {visual metrics}.  Similarly to the case of boundaries of hyperbolic groups, its Ahlfors-regular conformal dimension is defined as the infimum of the Hausdorff dimensions of Ahlfors-regular metrics quasisymmetric to a visual metric.

Section 4 discusses the \emph{critical exponent} $p_c(G, \bim)$ for contracting self-similar groups as mentioned above.   We prove that the critical exponent $p_c(G, \bim)$ of a contracting group is not larger than the conformal dimension $\cdim(G, \bim)$, see Theorem \ref{th:critexp}.
Lower estimates for $p_c(G, \bim)$, and hence also for $\cdim(G, \bim)$, can be obtained from exhibiting a carefully chosen sequence of group elements. In Theorem \ref{th:thurstonmap}, we show a relation between $\ell^p$-contraction and the spectral radius of the linear operator $T_p$, where $T_p$ is the Thurston linear transform adapted to $\ell^p$ calculations.  

Section 5 is devoted to random walk, and contains the proof of the main result on the Liouville property. It starts with a brief review of preliminaries on random walk and electric network theory, bounded harmonic functions, and the entropy criterion. We then show a direct bound on $\ell^q$-capacity on the Schreier graphs of the action of a contracting group $(G, \bim)$ on the levels of the tree, which is valid for $q>\cdim(G, \bim)$  and for all symmetric measures which admit a finite moment of order $p>p_c(G, \bim)$, see Proposition \ref{pr:energyestimateforg}, from which the estimates on capacity stated in Theorem \ref{th:energies} follow. We then give a proof of the Liouville property for the finite support case based on traverse counting and $\ell^2$-capacity estimates. For the general case under finite $p$-moment, $p>p_c(G, \bim)$, we first apply a geometric argument similar to the proof of Theorem \ref{th:traverses}, then random walk entropy is bounded by contributions from traverses and long jumps separately. The former is controlled by $\ell^2$-capacity estimates in the same way as in the finite supported case; and the latter is controlled by $\ell^p$-contraction. 

Section 6 contains several examples of applications of the main theorem. The first class of examples are groups generated by automata of \emph{polynomial activity growth}, as defined by S.~Sidki in~\cite{sid:cycl}. It was shown in~\cite{nekpilgrimthurston} that limit spaces of all contracting self-similar groups in this class have Ahlfors-regular conformal dimension 1.  Amenability of such groups also follows (as explained in~\cite{nekpilgrimthurston}) from the theory of extensively amenable actions, developed in~\cite{YNS,JMMS:extensive} (the  Liouville property for all groups in this class seems to be new even for finitely supported symmetric measures).

The second class of examples is iterated monodromy groups of sub-hyperbolic complex rational functions. If $f$ is a sub-hyperbolic rational function, then the limit space of its iterated monodromy group $\img{f}$ is canonically homeomorphic to the Julia set of $f$. Moreover, the spherical metric restricted to the Julia set, seen as a metric on the limit space of the group, is quasi-conformal and Ahlfors-regular. In particular, the Hausdorff dimension of the Julia set is an upper bound of the conformal dimension of the limit space of the iterated monodromy group. 
It follows that if the Julia set of the function is not the whole sphere, then the iterated monodromy group is Liouville (for every symmetric measure with finite second moment) and thus amenable.

We provide an explicit example of the iterated monodromy group of a rational function whose Julia set is homeomorphic to the Sierpi\'nski carpet (and thus can not be generated by an automaton of polynomial activity growth).

We also describe explicitly a contracting self-similar group whose  limit space coincides with the classical Sierpi\'nski carpet (i.e., quasi-symmetric to it).

Finally, we give an example of a finitely generated self-similar contracting group, which is not Liouville for any finitely supported non-degenerate measure.

\subsection{Acknowledgements}
This work was initiated during the trimester {\it Groups acting on fractals} at IHP from April to June 2022. The authors thank the organizers of the program and the Institut for the support and hospitality. We thank two anonymous referees for their very careful reading and helpful comments.

\section{Self-similar groups and their limit spaces}
\subsection{Basic definitions}

Let $\alb$ be a finite alphabet. The set of finite words $\alb^\ast$ over $\alb$ is naturally identified with the vertex set of a rooted tree, with root the empty word $\varnothing$, where each $w\in \alb^\ast$ is connected by an edge to $wx$ for every $x\in \alb$. We say that $w\in \alb^*$ is \emph{below} $v\in \alb^*$ (or a \emph{descendant} of $v$) if $w=vu$ for some $u\in \alb^\ast$.
\begin{defn}
An action of a group $G$ on $\xs$ is said to be \emph{self-similar} if for every $g\in G$ and $v\in\xs$, there exists an element $g|_v\in G$ such that
\[g(vw)=g(v)g|_v(w)\]
for all $w\in\xs$.
\end{defn}

 Note that a self-similar action is in particular an action by automorphisms on the rooted tree $\alb^\ast$. 

Consider a faithful self-similar action of $G$ on $\xs$. The \emph{biset of the action} is the set $\bim$ of transformations of $\xs$ of the form 
\[x\cdot g\colon v\mapsto xg(v)\]
for $x\in\alb$ and $g\in G$. Since the action is faithful, the set $\bim$ is in a natural bijection with the set $\alb\times G$.

The set $\bim$ is closed under pre- and post-compositions with elements of $G$, and hence it is a \emph{biset}, i.e., a set with commuting left and right actions of $G$. They are given by the rules
\[(x\cdot g)\cdot h=x\cdot gh,\qquad h\cdot(x\cdot g)=h(x)\cdot h|_xg.\]

The right action of $G$ on the biset $\bim$ is free and has $|\alb|$ orbits, where $\alb\cdot\{1\}\subset\bim$ (naturally identified with $\alb$) intersects each orbit exactly once.

In general, a \emph{self-similar group} $(G, \bim)$ is a group $G$ and a biset $\bim$ such that the right action is free and has finitely many orbits. A subset $\alb\subset\bim$ is called a \emph{basis} of the biset if it intersects each orbit exactly once.

If $\bim_1, \bim_2$ are two $G$-bisets, then the biset $\bim_1\otimes\bim_2$ is the quotient of $\bim_1\times\bim_2$ by the identification $(x\cdot g, y)\sim (x, g\cdot y)$ and with the actions on the quotient induced by the actions $g\cdot (x, y)\cdot h=(g\cdot x, y\cdot h)$. We denote the image of the element $(x, y)\in\bim_1\times\bim_2$ in $\bim_1\otimes\bim_2$ by  $x\otimes y$ or just by $xy$.

This operation is associative up to isomorphisms of bisets. In particular, for every biset $\bim$ we get bisets $\bim^{\otimes n}$ for all $n\ge 0$. Here $\bim^{\otimes 0}$ is defined as $G$ with the natural left and right $G$-actions.

If $\alb\subset\bim$ is a basis of a biset, then every element of $\bim$ is written in a unique way as $x\cdot g$ for $x\in\alb$ and $g\in G$. For every $g\in G$ and every $x\in\alb$ there exist unique $y\in\alb$ and $h\in G$ such that
\[g\cdot x=y\cdot h.\]
It is easy to see that the map $x\mapsto y$ is a permutation of $\alb$ (since the left $G$-action permutes the right $G$-orbits). We get the \emph{associated action} of $G$ on $\alb$, and denote $y=g(x)$, $h=g|_x$. The map $g\mapsto (\sigma_g, (g|_x)_{x\in\alb})\in\mathop{Symm}(\alb)\ltimes G^{\alb}$, where $\sigma_g\in\mathop{Symm}(\alb)$ is the image of $g$ under the associated action on $\alb$, is called the \emph{wreath recursion}. It is a homomorphism from $G$ to the wreath product $\mathop{Symm}(\alb)\ltimes G^{\alb}$.

Changing the basis $\alb$ corresponds to composing the wreath recursion with an inner automorphism of the wreath product.

If $\alb$ is a basis of $\bim$, then $\alb^{\otimes n}$ is a basis of $\bim^{\otimes n}$. We identify the set $\alb^{\otimes n}$ with $\alb^n$ in the natural way. Every element of $\bim^{\otimes n}$ is written in a unique way as $v\cdot g$ for $v\in\alb^n$ and $g\in G$.

In particular, for every $v\in\alb^n\subset\bim^{\otimes n}$ and for every $g\in G$ there exist unique $u\in\alb^n$ and $g|_v\in G$ such that 
\[g\cdot v=u\cdot g|_v.\]
The map $v\mapsto u$ is a permutation of $\alb^n$ (conjugate to the permutation induced by $g$ on the set of right orbits of $\bim^{\otimes n}$), so we denote $u=g(v)$.

We get hence the \emph{associated action} of $G$ on $\xs$. 
The associated action on $\xs$ does not depend, up to conjugacy, on the choice of the basis $\alb$ (since it is conjugate to the left action of $G$ on $\bigcup_{n\ge 0}\bim^{\otimes n}/G$). The \emph{faithful quotient} of a self-similar group $(G, \bim)$ is the quotient of $G$ by the kernel of the associated action on $\xs$.

Note that if $(G, \bim)$ is a self-similar group and $m\ge 1$, then $(G, \bim^{\otimes m})$ is a self-similar group with the same faithful quotient. Indeed the associated action of $G$ on the tree $(\alb^m)^*$ identifies with the restriction of the original action of $G$ on the levels of $\alb^\ast$ that are multiples of $m$. For many arguments, there is no loss of generality in replacing $(G,\bim)$ by some iterate $(G, \bim^{\otimes m})$, as we shall often do.

\subsection{Contracting groups and limit space}

Let $(G, \bim)$ be a self-similar group. Let $\alb\subset\bim$ be a basis. The self-similar group is said to be \emph{contracting} if there exists a finite set $\nuke\subset G$ such that for every $g\in G$ there exists $n$ such that 
\[g|_v\in\nuke\]
for all $v\in\alb^m$ for all $m\ge n$. The smallest set $\nuke$ satisfying this condition is called the \emph{nucleus} of the group. It is shown in~\cite[Corollary~2.11.7]{nek:book} that the property to be contracting does not depend on the choice of the basis $\alb$. The nucleus, however, depends on $\alb$.

\begin{defn}
Let $(G, \bim)$ be a contracting group. Choose a basis $\alb\subset\bim$, and let $\xmo$ be the space of left-infinite sequences $\ldots x_2x_1$ with the direct product topology (where $\alb$ is discrete). The \emph{limit space} $\lims$ is the quotient of $\xmo$ by the \emph{asymptotic equivalence relation}, identifying two sequences $\ldots x_2x_1, \ldots y_2y_1$ if and only if there exists a sequence $g_n$ of elements of the nucleus $\nuke$ such that $g_n(x_n\ldots x_2x_1)=y_n\ldots y_2y_1$.
\end{defn}

It is shown in~\cite{nek:book} that the asymptotic equivalence relation is transitive, and that the limit space is a metrizable compact space of finite topological dimension. 

It follows from the definition that the asymptotic equivalence relation is invariant under the shift $\ldots x_2x_1\mapsto \ldots x_3x_2$. Consequently, the shift induces a continuous map $\si\colon \lims\arr\lims$, which we call the \emph{limit dynamical system}. It is shown in~\cite{nek:book} that the limit dynamical system does not depend, up to topological conjugacy, on the choice of the basis $\alb$.

\begin{defn}
Let $(G, \bim)$ and $\alb$ be as above. The \emph{limit $G$-space} $\limg$ is the quotient of the direct product $\xmo\times G$ (where $G$ is discrete) by the equivalence relation identifying $\ldots x_2x_1\cdot g$ and $\ldots y_2y_1\cdot h$ if and only if there exists a sequence $g_n$ of elements of the nucleus such that
\[g_n\cdot x_n\ldots x_2x_1\cdot g=y_n\ldots y_2y_1\cdot h\] in $\bim^{\otimes n}$ for all $n$.
\end{defn}

The equality in the definition of $\limg$ is equivalent to the equalities
\[g_n(x_n\ldots x_2x_1)=y_n\ldots y_2y_1,\qquad g_n|_{x_n\ldots x_2x_1}g=h.\]

We have a natural right action of $G$ on $\limg$ induced by the action $(\ldots x_2x_1\cdot g)\cdot h=\ldots x_2x_1\cdot gh$ on $\xmo\times G$. One can show that the space $\limg$ and the action do not depend (up to topological conjugacy) on the choice of the basis $\alb$.

\begin{defn}
We denote by $\Tcal$ the image of $(\xmo, 1_G)\subset\xmo\times G$ in $\limg=\xmo\times G/\sim$. 
\end{defn}

The action of $G$ on $\limg$ is proper and co-compact with the quotient $\limg/G$ naturally homeomorphic to the limit space $\lims$. The set $\Tcal$ is compact and intersects every $G$-orbit. 



\begin{defn}
For $v\in\xs$, we denote by $\Tcal_v$ the image of the set $\xmo v$ of sequences ending with $v$ in the limit space $\lims$. We call it a \emph{tile of $n$th level} if $n=|v|$. In particular, $\lims$ is the tile of $0$th level.
\end{defn}

For every point $\xi\in\lims$ and every $n\ge 0$, the union of the tiles of the $n$-th level containing $\xi$ is a closed neighborhood of $\xi$. The following is \cite[Proposition 3.5.8]{nek:book}.

\begin{prop}
\label{pr:intersectingtiles}
Two tiles $\Tcal_v, \Tcal_u$ of the same level intersect if and only if there exists an element $g$ of the nucleus such that $g(v)=u$.
\end{prop}
\subsection{Contracting models}

Let $(G, \bim)$ be a contracting self-similar group, and let $\X$ be a proper metric space on which $G$ acts properly and co-compactly by isometries from the right.

Then $\X\otimes\bim$ is defined as the quotient of $\X\times\bim$ (where $\bim$ is discrete) by the action $(\xi, x)\mapsto (\xi\cdot g^{-1}, g\cdot x)$. We denote the orbit of $(\xi, x)$ by $\xi\otimes x$. We have then $\xi\otimes g\cdot x=\xi\cdot g\otimes x$ for all $\xi\in\X, x\in\bim$, and $g\in G$.

Let $\alb$ be a basis of $\bim$. Assume that $I\colon \X\otimes\bim\arr\X$ is a $G$-equivariant map. Then $I$ is  uniquely determined by the maps $I_x(\xi)=I(\xi\otimes x)$, which satisfy the following condition:
\begin{equation}
    \label{eq:Iequivariance}
    I_x(\xi\cdot g)=I_{g(x)}(\xi)\cdot g|_x,
\end{equation}
for all $x\in\alb$, $g\in G$, and $\xi\in\X$, where $g(x)\in\alb$ and $g|_x\in G$, as usual, are given by $g\cdot x=g(x)\cdot g|_x$.

Conversely, every collection of maps $I_x\colon \X\arr\X$ satisfying~\eqref{eq:Iequivariance} defines an equivariant map $I\colon \X\otimes\bim\arr\X$ by the formula $I(\xi\otimes x\cdot g)=I_x(\xi)\cdot g$.

Every equivariant map $I\colon \X\otimes\bim\arr\X$ induces an equivariant map $I_n\colon \X\otimes\bim^{\otimes (n+1)}\arr\X\otimes\bim^{\otimes n}$ by the rule
\[I_n(\xi\otimes x\otimes v)=I(\xi\otimes x)\otimes v\]
for $\xi\in\X$, $x\in\bim$, $v\in\bim^{\otimes n}$.
According to this, we will sometimes denote $I$ by $I_0$.

We also denote by $I^n\colon \xi\otimes\bim^{\otimes n}\arr\X$ the composition $I_{n-1}\circ I_{n-2}\circ\cdots \circ I_1\circ I$.

\begin{defn}
\label{def:contractingGmap}
A $G$-equivariant map $I\colon \X\otimes\bim\arr\X$ is said to be \emph{contracting} if there exist $n\ge 1$ and $\lambda\in (0, 1)$ such that $d(I^n(\xi_1\otimes v), I^n(\xi_2\otimes v))\le\lambda d(\xi_1, \xi_2)$ for all $\xi_1, \xi_2\in\X$ and $v\in\bim^{\otimes n}$.
\end{defn}

A $G$-equivariant map $I\colon \X\otimes\bim\arr\X$ is contracting if and only if there exist $C>0$ and $\lambda\in (0, 1)$ such that
\[ d(I^n(\xi_1\otimes v), I^n(\xi_2\otimes v))\le C\lambda^n d(\xi_1, \xi_2) \]
for all $n$, $v\in\bim^{\otimes n}$, $\xi_1, \xi_2\in\X$.

The natural map $\limg\otimes\bim\arr\limg$ induced by the map $(\ldots x_2x_1\cdot g, x\cdot h)\mapsto \ldots x_2x_1g(x)\cdot g|_xh$ from $\xmo\times G\times \bim$ to $\xmo\times G$ is a well-defined $G$-equivariant \emph{homeomorphism}.  We will denote the image of $\xi\otimes x\in\limg\otimes\bim$ just by $\xi\otimes x$, thus identifying $\limg\otimes\bim$ with $\limg$ (this slight abuse of notation is justified by the fact that $x_1\cdot g\otimes x\cdot h\mapsto x_1\otimes g(x)\cdot g|_xh$ is the identical map on $\bim^{\otimes 2}$).

It is proved in~\cite[Theorem~2.10]{nek:dimension}  that the natural homeomorphism $\limg\otimes\bim\arr\limg$ is contracting in the sense of Definition~\ref{def:contractingGmap} (with respect to a natural distance on $\limg$) provided $G$ is finitely generated.

The following theorem is proved in~\cite{nek:models}.

\begin{thm}
\label{th:contrmodels}
Let $I\colon \X\otimes\bim\arr\X$ be a contracting $G$-equivariant map. Then the inverse limit of the sequence of $G$-spaces $\X\otimes\bim^{\otimes n}$ and maps $I_n$ is $G$-equivariantly homeomorphic to the limit $G$-space $\limg$. The map on the inverse limit induced by the maps $I_n$ is conjugated by the homeomorphism with the natural homeomorphism $\limg\otimes\bim\arr\limg$.
\end{thm}

\begin{defn}
Let $(G, \bim)$ be a contracting self-similar group. A \emph{contracting model} of $(G, \bim)$ is a proper co-compact action of $G$ on a metric space $\X$ by isometries together with a contracting equivariant map $I\colon \X\otimes\bim\arr\X$.
\end{defn}

It is proved in~\cite{nek:models} (see also~\cite[Theorem~3.1]{nek:dimension}) that for every finitely generated contracting group $(G, \bim)$ there exists $n_0$ such that $(G, \bim^{\otimes n_0})$ has a contracting model $I\colon \X\otimes\bim^{\otimes n_0}\arr\X$ where $\X$ is a connected simplicial complex and the maps $I_x$ are piecewise affine.

 We shall say that a metric space $(\X, d)$ is \emph{almost geodesic} if there exists $C>0$ such that for every $\xi, \eta\in \X$, there exists a sequence of points $x_0=\xi, x_1, \ldots, x_n=\eta$ such that $d(x_i, x_{i+1})\leq C$ and 
\[\lvert d(\xi, \eta)-\sum_{i=0}^{n-1} d(x_i, x_{i+1})\rvert \leq C.\]
It is easy to see (and it follows, e.g., from Theorem 4.5.C and Corollary 4.B.11 in \cite{Cor-Har-LC}) that a proper almost geodesic space $\X$ satisfies the Schwarz--Milnor lemma\footnote{In the most common formulation of the Schwarz--Milnor lemma, $\X$ is supposed to be a length space in the usual sense. Note also that an almost geodesic space is in particular \emph{large-scale geodesic} in the sense of \cite[Definition 3.B.1]{Cor-Har-LC}, a condition that already suffices for the Schwarz--Milnor lemma.} : if a group $G$ acts properly, isometrically and co-compactly on $\X$ from the right, then $G$ is finitely generated and for any word metric on $G$ and every $x\in \X$, the orbital map $g\mapsto x\cdot g$ is a quasi-isometry from $G$ to $\X$. 

We will need the following simple lemma.
\begin{lem}
\label{lem:lengthofsections}
Suppose that $G$ is finitely generated, and let $l(g)$ denote the word length of an element $g\in G$ with respect to a fixed finite generating set of $G$. Choose a basis $\alb\subset\bim$. Let $\X$ be a proper metric space on which $G$ acts by isometries from the right so that the action is co-compact and proper, and  let $I\colon \X\otimes\bim\arr\X$ be a $G$-equivariant map. 
\begin{enumerate}[label=(\roman*)]
\item \label{i-attractor} If $I$ is contracting, then for every compact subset $\mathcal{T}_0\subset \X$, there exists $\mathcal{T}\supset \mathcal{T}_0$ compact such that $I^{n}(\xi\otimes v)\in \mathcal{T}$ for every $n\ge 1, \xi\in \mathcal{T}$ and $v\in \alb^n$.

\item \label{i-SM} Suppose that $\X$ is an almost geodesic space, and let $\mathcal{T}$ be any compact set satisfying the conclusion of item \ref{i-attractor}. Then there exists a constant $C>0$ such that 
\[l(g|_v)\le Cd(I^n(\xi\otimes v), I^n(\xi\cdot g\otimes v))+C\]
for every $n\ge 1$, $g\in G$, $\xi\in\Tcal$, and $v\in\alb^n$.

\end{enumerate}


\end{lem}


\begin{proof}
 \ref{i-attractor} We take $\mathcal{T}$ to be the closure of the set of points of the form $I^{|v|}(\xi\otimes v)$ for $\xi\in \mathcal{T}_0$  and $v\in \xs$ (compare with the proof of \cite[Theorem 5.9]{nek:models}). Let us check that it has finite diameter (and hence it is compact). Fix $m\ge 1$ and $\lambda\in (0, 1)$ such that $d(I^m(\xi_1\otimes v), I^m(\xi_2 \otimes v))\leq \lambda d(\xi_1, \xi_2)$ for all $v\in \bim^{\otimes m}$.  For $k\ge 1$, let $\mathcal{T}_k$ be the set of points of the form $I^{|v|}(\xi\otimes v)$ for $\xi\in \mathcal{T}_0$ and $\lvert v\rvert\leq mk$, and note that $\mathcal{T}_k\subset \mathcal{T}_{k+1}$ and that $\mathcal{T}$ is the closure of $\bigcup_k \mathcal{T}_k$. Pick a basepoint $\xi_0\in \mathcal{T}_0$, and let $C>0$ be an upper bound of the diameter of $\mathcal{T}_1$. For $k\ge 1$,  if $\xi\in \mathcal{T}_{k+1}\setminus \mathcal{T}_k$, then $\xi=I^m(\xi' \otimes v)$ for some $\xi'\in \mathcal{T}_k, v\in \alb^m$, and  we have
 \[d(\xi, \xi_0)\leq d(\xi_0, I^m(\xi_0\otimes v))+ d(I^m(\xi' \otimes v), I^m(\xi_0\otimes v))\leq C+\lambda d(\xi', \xi_0),\]
where we have used that $\xi_0, I^m(\xi_0\otimes v)\in \mathcal{T}_1$. From this, it immediately follows by induction that $\mathcal{T}_k$ is contained in the ball centered at $\xi_0$ of radius $C(1+\lambda+\cdots +\lambda^{k-1})$, and hence $\mathcal{T}$ has diameter bounded by $2C/(1-\lambda)$. 

\ref{i-SM} Since $\X$ is an almost geodesic space, the Schwarz--Milnor lemma and compactness of $\mathcal{T}$ imply that there exists $C>0$ such that $l(g)\le Cd(\xi_1, \xi_2\cdot g)+C$ for all $\xi_1, \xi_2\in\Tcal$. 

It follows that $l(g|_v)\le Cd(I^n(\xi\otimes v), I^n(\xi\otimes g(v))\cdot g|_v)+C=Cd(I^n(\xi\otimes v), I^n(\xi\cdot g\otimes v)+C$.
\end{proof}

\subsection{Iterated monodromy groups}
\label{ss:img}

\begin{defn}
A \emph{virtual endomorphism} of a topological space $\mathcal{J}$ is a finite degree covering map $f\colon \mathcal{J}_1\arr\mathcal{J}$ together with a continuous map $\phi\colon \mathcal{J}_1\arr\mathcal{J}$.
\end{defn}

Let $f, \phi\colon \mathcal{J}_1\arr\mathcal{J}$ be a virtual endomorphism. Suppose that $\mathcal{J}$ is path connected. Consider the fundamental group $G=\pi_1(\mathcal{J}, t)$. The \emph{associated biset} is the set $\bim$ of pairs $(z, [\delta])$, where $z\in f^{-1}(t)$, and $[\delta]$ is the homotopy class of a path $\delta$ in $\mathcal{J}$ connecting $t$ to $\phi(z)$. The action of $G$ on $\bim$ is given by the formulas
\[(z, [\delta])\cdot[\gamma]=(z, [\delta\gamma]),\qquad [\gamma](z, [\delta])=(y, [\phi(\gamma_z)\delta]),\]
where $\gamma$ is a loop based at $t$, $\gamma_z$ is the lift of $\gamma$ by $f$ starting in $z$, $y$ is the end of $\gamma_z$. We multiply paths as functions: in a product $\alpha\beta$ the path $\beta$ is passed before $\alpha$.

The \emph{iterated monodromy group} of the virtual endomorphism is the faithful quotient of the self-similar group $(G, \bim)$.

By the definition of the associated biset, a basis of $\bim$ is a collection \[\{(z_1, [\ell_1]), (z_2, [\ell_2]), \ldots, (z_n, [\ell_n]\},\] where $\{z_1, z_2, \ldots, z_n\}=f^{-1}(t)$, and $\ell_i$ is a path from $t$ to $\phi(z_i)$. Let us denote $x_i=(z_i, [\ell_i])$. Then, as shown in \cite[Theorem 4.5]{nek:models}, the action of the iterated monodromy group on the tree $\xs$ for $\alb=\{x_1, x_2, \ldots, x_n\}$ is given by the recurrent formula
\begin{equation}
    \label{eq:imgr}
[\gamma](x_iw)=x_j[\ell_j^{-1}\phi(\gamma_i)\ell_i](w),
\end{equation}
where $\gamma_i$ is the unique lift of $\gamma$ by $f$ starting in $z_i$, and $j$ is such that $z_j$ is the end of $\gamma_i$.

A particular case of this situation is a \emph{partial self-covering}, i.e., a finite degree covering map $f\colon \mathcal{J}_1\arr\mathcal{J}$ where $\mathcal{J}_1\subset\mathcal{J}$. It is a virtual endomorphism with the map  $\phi\colon \mathcal{J}_1\arr\mathcal{J}$ equal to the identical embedding. Then the \emph{iterated monodromy group} of $f$ is the iterated monodromy group of the virtual endomorphism $f, \phi\colon \mathcal{J}_1\arr\mathcal{J}$. In this case, it is natural to omit $\phi$ (as it is the identity map) in the formula~\eqref{eq:imgr}.

An important class of examples consists of \emph{sub-hyperbolic complex rational functions}. 

Let $f(z)\in\mathbb{C}(z)$ be a rational function seen as a self-map of the Riemannian sphere $\mathbb{PC}^1$. Denote by $P_f$ the closure of the union of the forward orbits of the \emph{critical values} of $f$ (i.e., values at the critical points of the map $f\colon \mathbb{PC}^1\arr\mathbb{PC}^1$). 

The function $f$ is said to be \emph{sub-hyperbolic} if it is expanding with respect to some orbifold metric on a neighborhood of its Julia set, see~\cite[\S 19]{milnor}. If such a metric can be chosen to be a usual Riemannian metric (i.e., with no non-trivial orbifold singularities), then $f$ is said to be \emph{hyperbolic}. By a theorem of A.~Douady and J.H.~Hubbard (see~\cite[Theorem~19.6]{milnor}) a rational function $f$ is sub-hyperbolic if and only if the orbit of every critical point is either finite or converges to an \emph{attracting cycle},  i.e., a periodic orbit $f\colon x_1\mapsto \cdots \mapsto x_d\mapsto x_1$ such that $\prod_{i} |f'(x_i)|<1$ (it is hyperbolic if and only if all critical points converge to an attracting cycle). In particular, $f$ is sub-hyperbolic provided $P_f$ is finite (such functions are called \emph{post-critically finite}). 
 A sub-hyperbolic function induces a partial self-covering
\[f\colon \mathbb{PC}^1\setminus f^{-1}(P_f)\to \mathbb{PC}^1\setminus P_f.\]
 (Note that $P_f$ and $f^{-1}(P_f)$ are countable closed sets with finitely many accumulation points, in particular their complements are connected manifolds). By definition, the iterated monodromy group  $\img{f}$ is the iterated monodromy group of this partial self-covering.

The following result is a  consequence of the orbifold version of Theorem~\ref{th:contrmodels} (see~\cite[Theorem~5.5.3]{nek:book} and \cite[Theorem~4.5.36]{nek:dyngroups}). 
\begin{thm}
\label{th:imgandlimspace}
Let $f$ be a sub-hyperbolic (e.g., post-critically finite) complex rational function. Then its iterated monodromy group $G=\img{f}$ is contracting and its limit dynamical system $\si\colon \lims\arr\lims$ is topologically conjugate to the restriction of $f$ to its Julia set.
\end{thm}
\begin{proof}[Proof sketch]
If the function $f$ is post-critically finite, this is exactly  \cite[Theorem 6.4.4]{nek:book}. We briefly explain how to apply \cite[Theorem~5.5.3]{nek:book} to a general sub-hyperbolic function, since that result assumes that the fundamental group of the base space is finitely generated, and this is not the case of $\mathbb{PC}^1\setminus P_f$ when $P_f$ is infinite. This issue can be bypassed as follows.
Fix finitely many small open discs  $D_1, \cdots, D_r\subset \mathbb{PC}^1$, with each $\overline{D}_i$ contained in the basin of attraction of some attracting cycle, whose union $D_1\cup \cdots \cup D_k$  contains all points in the attracting cycles and all post-critical points converging to them, and is forward $f$-invariant. One can find such disks $D_i$, for example, using the \emph{Fatou coordinate} uniformizing the Fatou components~\cite[\S 9]{milnor}, see for example the proof of~\cite[Theorem~6.6]{nek:cpalg}. Then the complement $\mathcal{M}$ of $D_1 \cup \ldots\cup  D_r$ is a compact neighbourhood of the Julia  set such that  $f^{-1}(\mathcal{M})\subset\mathcal{M}$. There is a natural structure of \emph{Thurston orbifold} and an orbifold metric on $\mathcal{M}$ with respect to which $f\colon f^{-1}(\mathcal{M})\arr\mathcal{M}$ is a uniformly expanding partial self-covering, see~\cite[\S 19]{milnor} and~\cite[\S~4.7.1]{nek:dyngroups}. Let us denote  by $\img{f, \mathcal{M}}$ its iterated monodromy group.  Theorem \cite[Theorem~5.5.3]{nek:book} (or alternatively \cite[Theorem 5.10]{nek:models}) can be applied to this new self-covering, and implies that the group $\img{f, \mathcal{M}}$ is contracting and its limit dynamical system is conjugate to the restriction of $f$ to its Julia set.    

If we choose the basepoint and the connecting paths $\ell_x$ inside $\mathcal{M}$ and compute the iterated monodromy groups using~\eqref{eq:imgr}, then 
the group $\img{f, \mathcal{M}}$ will be realized as a subgroup of $\img{f}$, acting on the same rooted tree $\alb^\ast$. For every compact set $C\subset\mathbb{PC}^1\setminus P_f$ there exists $n$ such that $f^{-n}(C)\subset\mathcal{M}$ (since each $f^n(\overline{D}_i)$ is disjoint from $C$ for $n$ sufficiently large). Consequently \eqref{eq:imgr} implies that for every element $g\in\img{f}$ there exists $n$ such that $g|_v\in\img{f, \mathcal{M}}$ for all $v\in\alb^n$. In particular, the self-similar group $\img{f}$ is contracting with the same nucleus as $\img{f, \mathcal{M}}$, and hence it also has the same limit dynamical system. \qedhere

\end{proof}

\subsection{One-dimensional limit spaces}\label{subsec:one-dim}

The following theorem is proved in~\cite{nek:dimension}.

\begin{thm}
\label{th:onedim}
Suppose that the limit space $\lims$ (equivalently the limit $G$-space $\limg$) of a contracting finitely generated group $(G, \bim)$ has topological dimension 1. Then there is $n_0$ and a contracting model $I\colon \Delta_0\otimes\bim^{\otimes n_0}\arr\Delta_0$, where $\Delta_0$ is a locally finite connected graph on which $G$ acts properly and co-compactly by automorphisms.
\end{thm}
Below we suppose to be in the situation of Theorem \ref{th:onedim}. After replacing $\bim$ by $\bim^{\otimes n_0}$, we assume that $n_0=1$.

Denote $\Delta_n=\Delta_0\otimes\bim^{\otimes n}$. Since the action of $G$ on $\Delta_0$ is by automorphisms of graphs, the spaces $\Delta_n$ are also graphs. 

Denote by $\Gamma_n$ the quotient $\Delta_n/G$ and by $\pi_n\colon \Delta_n\to \Gamma_n$ the quotient projection. Then $\Gamma_n$ is a finite graph. The maps $I_n\colon \Delta_{n+1}\arr\Delta_n$ induce continuous maps $\phi_n\colon \Gamma_{n+1}\arr\Gamma_n$. We may assume that in the original action of $G$ on $\Delta_0$, no pair of neighbouring vertices of $\Delta_0$ belong to the same $G$-orbit (otherwise pass to the barycentric subdivision of $\Delta_0$). In particular, $\Gamma_0$ has no loops.

\begin{notation} \label{notation:traverses} Denote by $U_z$ the ball of radius 1/3 with center in a vertex $z$ of the graph $\G_0=\Delta_0/G$.  It is a bouquet of segments of length 1/3. (We use the combinatorial distance in the graphs, identifying each edge with a real segment of length 1.)

Fix a basepoint $\xi_0\in\Delta_0$. 
Denote by $U_{z, n}$ the set of elements $v\in\alb^n$ such that $\pi_0(I^n(\xi_0\otimes v))\in U_z$. 

\end{notation}

\begin{defn} \label{d-traverses}
Let $s_ts_{t-1}\ldots s_1$ be a word in elements $s_i\in G$ (formally, an element of the free group with basis  $G$). Let $z_1, z_2$ be vertices of $\Gamma_0$. A \emph{traverse of level $n$ from $z_1$ to $z_2$} of the word $s_ts_{t-1}\ldots s_1$ is a quadruple $(i, j, v_1, v_2)$, where $1\le i\le j\le t$,  $v_1\in U_{z_1, n}$, $v_2\in U_{z_2, n}$ for $z_1\ne z_2$ are such that $s_k\cdots s_{i+1}s_i(v_1)\notin U_{z_1, n}\cup U_{z_2, n}$ for all $i<k<j$, and $s_js_{j-1}\cdots s_i(v_1)=v_2$.
\end{defn}

In other words, a traverse is a segment of the word $s_ts_{t-1}\ldots s_1$ corresponding to a travel from a point of $U_{z_1, n}$ to a point of $U_{z_2, n}$ such that the sets $U_{z_1, n}$ and $U_{z_2, n}$ are not touched during the travel.

We denote by $\tau_n(s_ts_{t-1}\ldots s_1)$ the total number of traverses of level $n$ for the word $s_ts_{t-1}\ldots s_1$.

\begin{thm}
\label{th:traverses} Retain all previous assumptions and notation. Let $S$ be a finite generating set for $G$.
There exists $C>1$ such that for all $n$ big enough and every element $g=s_ts_{t-1}\cdots s_1\in G$, where $s_i\in S$, we have
\[\sum_{v\in\alb^n} l(g|_v)\le C\tau_n(s_ts_{t-1}\cdots s_1)+C|\alb^n|.\]
\end{thm}

\begin{proof}
By item~\ref{i-attractor} in Lemma~\ref{lem:lengthofsections}, we may choose a compact set $\mathcal{T}\ni \xi_0$ such that $I^n(\xi\otimes v)\in \mathcal{T}$ for every $\xi\in \mathcal{T}$ and every $n, v\in \alb^n$. Then item \ref{i-SM} in the same lemma (which can be applied because a connected graph is an almost geodesic space) implies that there exists $C_1>0$ such that for every $g\in G$ and $v\in \alb^n$, the length of $g|_v$ is bounded from above by $C_1d(I^n(\xi_0\otimes v), I^n(\xi_0\cdot g\otimes v))+C_1$.
Since the $G$ action is isometric, the quantity   $d(\xi_0\cdot g, \xi_0 \cdot sg)$  is uniformly bounded for $s\in S$ and $g\in G$. Hence there exists $n_0$ such that $d(I^n(\xi_0\cdot g \otimes v),I^n(\xi_0\cdot sg \otimes v))<1/3$ for every $g\in G, s\in S$ and $v\in \alb^n, n\ge n_0$.  Fix $v$. Consider $g=s_t\cdots s_1$ as in the statement of the theorem, and set $\xi_i=\xi_0\cdot s_i\cdots s_1$. Then the distance between two consecutive points $I^n(\xi_i\otimes v)$ is at most 1/3, and 
\[l(g|_v)\leq C_1d(I^n(\xi_0\otimes v), I^n(\xi_t\otimes v))+C_1. \]
 Let $\delta_i$ be a geodesic curve joining $I^n(\xi_{i-1} \otimes v)$ to $I^n(\xi_{i} \otimes v)$ and $\gamma=\delta_t\cdots \delta_1$ be their concatenation (where the rightmost curve is passed first). Consider an edge $e$ which is crossed by $\gamma$, meaning that there is a subcurve $\gamma|_{[t_1, t_2]}$ such that $\gamma(t_1), \gamma(t_2)$ are the extreme points of the edge and $\gamma(s)$ is in the interior of $e$ for $s\in (t_1, t_2)$. For every such edge-crossing we can find $i<j$ such that $I^n(\xi_i\otimes v)$ and $I^n(\xi_j\otimes v)$ are at distance at most 1/3 from the two extreme points of $e$, and every  $I^n(\xi_k\otimes v)$ with $i<k<j$  is in the interior of $e$  at distance at least  1/3 from the extreme points. The distance between $d(I^n(\xi_0\otimes v), I^n(\xi_t\otimes v))$ is not larger than the total number of  edge crossings of $\gamma$ plus 2. 
 
 Now recall that the quotient projection $\pi_0\colon \Delta_0\to \Gamma_0$ maps edges homeomorphically to edges, and moreover
\[\pi_0(I^n(\xi_i\otimes v))=\pi_0(I^n(\xi_0 \otimes s_i\cdots s_1(v))\cdot s_i\cdots s_1|_v)=\pi_0(I^n(\xi_0 \otimes s_i\cdots s_1(v))).\]
We deduce that to every edge-crossing of  $\gamma$ we can associate (injectively) a traverse. Thus there is a constant $C>0$ such that
\[l(g|_v)\le C\tau_{n, v}(s_t\cdots s_1)+C,\]
where $\tau_{n,v}$ denotes the number  traverses at level $n$ of the form $(i, j, v_1, v_2)$ with  $v_1=s_i\cdots s_1(v)$ and $v_2=s_j\cdots s_1(v)$. Summing over $v$ we obtain the statement of the theorem. \qedhere
\end{proof}

\section{Conformal dimension of contracting groups}

\subsection{Visual metric}

Let $(G, \bim)$ be a contracting group. Choose a basis $\alb$ of $\bim$ and let $\nuke$ be the corresponding nucleus of $G$. The \emph{self-similarity complex} is the graph with set of vertices $\xs$ in which two vertices are connected by an edge in one of the following two situations: either they are of the form $v, xv$ for $v\in\xs$ and $x\in\alb$ (\emph{vertical edges}), or they are of the form $v, g(v)$ for $v\in\xs$ and $g\in\nuke$ (\emph{horizontal edges}).

\begin{notation}
    Given two vertices $x_n\ldots x_2x_1$ and $y_m\ldots y_2y_1$, we denote by $\ell(x_n\ldots x_2x_1, y_m\ldots y_2y_1)$  the largest $k\le\min(n, m)$ such that there exists $g\in\nuke$ such that $g(x_k\ldots x_2x_1)=y_k\ldots y_2y_1$.
\end{notation}

    Below we use the notation $(x | y)_o$ for the Gromov product in a metric space (with respect to a basepoint $o$), namely
    \[(x|y)_o=\frac{1}{2}(d(x, o)+d(y, o)-d(x, y)).\]
The following theorem is proved in~\cite{nek:hyplim} (see also~\cite[Theorem~3.8.8]{nek:book}). The statement about the Gromov product follows directly from the proof.

\begin{thm} \label{t-hyperbolicity}
The self-similarity complex is Gromov hyperbolic. Its boundary is homeomorphic to the limit space $\lims$, where a point of $\lims$ represented by $\ldots x_2x_1\in\xmo$ corresponds to the limit of the geodesic path $(x_1, x_2x_1, x_3x_2x_1, \ldots)$ in the self-similarity complex.

For every basepoint $o$, there exists a constant $C$ such that for any $v=x_n\ldots x_2x_1$ and $w=y_m\ldots y_2y_1$ we have
\[\lvert (v |  w)_\varnothing-(v | w)_o\rvert \leq C.\] 
\end{thm}

For $\xi, \eta\in \lims$, define $\ell(\xi, \eta)$ to be the  the largest $n$ such that there exist representatives $\xi=\cdots x_2x_1$ and $\eta=\cdots y_2y_1$ such that $g(x_n\ldots x_2x_1)=y_n\ldots y_2y_1$ for some $g\in\nuke$. It follows from Theorem \ref{t-hyperbolicity} that $\ell(\cdot, \cdot)$ coincides (up to an additive constant) with the natural extension of the Gromov product to $\lims$. Note that we define $\ell(\xi, \xi)=\infty$ for $\xi\in\lims$. 

The following is a classical result of the theory of Gromov-hyperbolic spaces, see \cite[\S3]{ghysdelaharpe}.

\begin{thm}
There exists $\alpha_0>0$ such that for every $\alpha<\alpha_0$ there exists a metric $d_\alpha$ on $\lims$ and a constant $C>1$ such that 
\begin{equation}
    \label{eq:visual}
C^{-1}e^{-\alpha\ell(\xi_1, \xi_2)}\le d_\alpha(\xi_1, \xi_2)\le Ce^{-\alpha\ell(\xi_1, \xi_2)}
\end{equation}
for all $\xi_1, \xi_2\in\lims$.
\end{thm}

Note that (for a fixed value of $\alpha$) the metric $d_\alpha$ is unique up to bi-Lipschitz equivalence. We call it the \emph{visual metric of exponent $\alpha$}.

If $f, g$ are real-valued functions with the same domain of definition, we write $f\asymp g$ if there exists a constant $C>0$ such that $C^{-1} f \leq g\leq C f$.

It follows from Proposition~\ref{pr:intersectingtiles} that for every $\zeta\in\lims$ the set \[T(\zeta, n)=\{\xi\in\lims : \ell(\xi, \zeta)\ge n\}\]
is equal to the union of the tiles $\Tcal_v$ of the $n$th level intersecting tiles of the $n$th level that contain $\zeta$. Note that $T(\zeta, n)$ is a union of not more than $|\nuke|^2$ tiles of the $n$th level, by Proposition \ref{pr:intersectingtiles}. For later use, we also record the following tightly related observation.
\begin{lem}\label{l-multiplicity-shadows}
    Fix $\zeta\in \lims$. Then for every $n\ge 1$, we have 
    \[|\{v\in \alb^n \colon \exists \xi\in \mathcal{T}_v \text{ such that }\zeta\in T(\xi, n)\}| \le |\nuke|^3.\]
\end{lem}
\begin{proof} Fix $u\in \alb^n$ such that $\zeta\in \mathcal{T}_u$. 
If $v\in \alb^n$ belongs to the set in the statement of the lemma, then there exists $\xi\in \mathcal{T}_v$ and $v_1, v_2\in \alb^n$ such that $\xi\in \mathcal{T}_v\cap \mathcal{T}_{v_1}, \mathcal{T}_{v_1}\cap \mathcal{T}_v\neq \varnothing$ and $\zeta\in \mathcal{T}_{v_2}\cap \mathcal{T}_u$. It follows from Proposition \ref{pr:intersectingtiles} that $v=h(u)$ for some $h\in \nuke^3$.  
\end{proof}

Note also that if $d_\alpha$ is a visual metric, then we have
\begin{equation}
    \label{eq:Tnvisual}
B_{d_\alpha}(\zeta, C^{-1}e^{-\alpha n})\subset T(\zeta, n)\subset B_{d_\alpha}(\zeta, Ce^{-\alpha n}),
\end{equation}
where $B_{d_\alpha}(x, R)$ denotes the ball of radius $R$ with center in $x$, and $C$ satisfies~\eqref{eq:visual}.

Let $\mu$ be the push-forward of the uniform Bernoulli measure on $\xmo$ by the quotient map $\xmo\arr\lims$. Since the quotient map is at most $|\nuke|$-to-one, we have, for every $v\in\alb^n$
\[|\alb|^{-n}\le\mu(\Tcal_v)\le |\nuke|\cdot|\alb|^{-n}.\]

It follows that $\mu(T(\zeta, n))\asymp |\alb|^{-n}$, so 
\[\mu(B_{d_\alpha}(\zeta, R))\asymp |\alb|^{-\frac{-\ln R}{\alpha}}=\exp\left(\frac{\ln R\ln|\alb|}{\alpha}\right)=R^{\frac{\ln|\alb|}{\alpha}}.\]

%

\subsection{Conformal dimension}
A homeomorphism $f\colon \X_1\arr\X_2$ between metric spaces $(\X_1, d_1)$ and $(\X_2, d_2)$ is said to be \emph{quasisymmetric} if there exists a homeomorphism $\eta\colon [0, \infty)\arr [0, \infty)$ such that 
\[\frac{d_2(f(x), f(y))}{d_2(f(x), f(z))}\le\eta\left(\frac{d_1(x, y)}{d_1(x, z)}\right)\] for all $x, y, z\in\X_1$ such that $x\neq z$.

We refer to \cite[Chapter 10]{Heinonen} for general properties of quasisymmetric homeomorphism. It is easy to see that the inverse of every quasisymmetric homeomorphism is also quasisymmetric \cite[Proposition 10.6]{Heinonen}. 

  Two metrics on the same underlying topological space are said to be quasisymmetric to each other if the identity map is quasisymmetric. 

\begin{defn}
\label{def:conformalmetric}
A metric $d$ on a space $\X$ is said to be \emph{Ahlfors-regular} if there exists $\beta>0$ and a measure $\mu$ on $\X$ such that 
\[\mu(B_d(x, R))\asymp R^\beta\]
for all $x\in\X$ and $R<\diam(\X)$. The exponent $\beta$  is equal to the Hausdorff dimension of $d$. 
\end{defn}

Let $(G, \bim)$ be a contracting self-similar group. 
It follows directly from the definitions that any two visual metrics $d_{\alpha_1}, d_{\alpha_2}$ on $\lims$ are quasisymmetric to each other.
 We have seen that a visual metric $d_\alpha$ is Ahlfors-regular of Hausdorff dimension $\frac{\ln|\alb|}{\alpha}$.

\begin{defn}
We say that a metric $d$ on $\lims$ is \emph{quasi-conformal} if it is quasi-symmetric to a visual metric.

The infimum of the set of Hausdorff dimensions of Ahlfors-regular quasi-conformal metrics on $\lims$ is called the \emph{Ahlfors-regular conformal dimension} of $\lims$ and is denoted by $\cdim\lims$ or $\cdim(G, \bim)$.
\end{defn}

\begin{rem}\label{r-top-dimension} It is a general fact that $\cdim \lims$ is always greater than or equal to the topological dimension of $\lims$ (see \cite[Theorem 1.4.12]{mackay_tyson}). On the other hand, the topological dimension of $\lims$ is zero if and only if $G$ is locally finite \cite[Theorem 4.1]{nek:dimension}. Consequently, for every finitely generated infinite contracting group $(G, \bim)$ we have $\cdim \lims \ge 1$.
\end{rem}






\begin{prop}
\label{pr:confdim}
Let $d$ be an Ahlfors--regular quasi-conformal metric on $\lims$. 
 \begin{enumerate}[label=(\roman*)]

\item \label{i-quasi-round} There exists  $C_1>0$ and $n_0$ such that for every $\xi \in \lims$ and every $n\ge n_0$, there exists $r$ such that $B_d(\xi,  r)\subset T(\xi, n)\subset B_d(\xi, C_1r)$.

\item \label{i-diam-exponential} There exist $C_2>0$ and $\lambda\in (0, 1)$ such that \[\sup_{\xi\in\lims}\diam_d(T(\xi, n))\le C_2\lambda^n\] for all $n$.
\end{enumerate}

\end{prop}

\begin{proof}
Let $\eta\colon [0, \infty)\arr[0, \infty)$ be a homeomorphism such that $\frac{d(x, y)}{d(x, z)}\le\eta\left(\frac{d_\alpha(x, y)}{d_\alpha(x, z)}\right)$ for all $x, y, z\in\lims$ with $x\neq z$, where $d_\alpha$ is a visual metric on $\lims$. Recall (see \eqref{eq:Tnvisual}) that there exists a constant $C>1$ such that 
\[B_{d_\alpha}(\xi, C^{-1} e^{-n\alpha})\subset T(\xi, n)\subset B_{d_\alpha}(\xi, C e^{-n\alpha}). \]
 In particular, the statement in \ref{i-quasi-round} is true for the metric $d_\alpha$, and it is easy to see that the same conclusion must hold for any quasisymmetric metric. Let us give a proof for completeness. Fix $\xi$ and $n$, and set $R\colon=2\sup\{d(\xi, x)\colon x\in T(\xi, n)\}$.
Then $T(\xi, n)\subset B_d(\xi, R)$, and we can choose $y\in T(\xi, n)$ such that $d(\xi, y)>\frac{1}{4} R$.
Suppose that $z\notin T(\xi, n)$. Then $d_\alpha(\xi, y)\leq Ce^{-n\alpha}$ and $d_\alpha(\xi, z)\geq C^{-1}e^{-n\alpha}$, so
\[\frac{d(\xi, y)}{d(\xi, z)}\leq \eta\left(\frac{d_\alpha(\xi, y)}{d_\alpha(\xi, z)}\right)\leq \eta(C^2),\]
and therefore 
\[d(\xi, z)\geq \eta(C^2)^{-1}d(\xi, y)\geq  (4\eta(C^2))^{-1}R.\]
Since this is true for every $z\notin T(\xi, n)$, it follows that 
\[B_d(\xi, C_1^{-1}R)\subset T(\xi, n)\subset B_d(\xi, R)\]
with $C_1=4\eta(C^2)$, which implies \ref{i-quasi-round}.

We now prove \ref{i-diam-exponential}.  Let $\xi\in \lims$ and $n, m \ge 0$. Then for every $x\in T(\xi, n+m)\subset\Tcal(\xi, n)$ there exists $y\in T(\xi, n+m)$ such that $2d(x, y)\ge \diam_d(T(\xi, n+m))$. Similarly, there exists $z\in T(\xi, n)$ such that $d_\alpha(x, z)\ge \diam_{d_\alpha}( T(\xi, n))/2$. Then
\[
    \frac{\diam_d( T(\xi, n+m))}{\diam_d( T(\xi, n))}\le \frac{2d(x, y)}{d(x, z)}\le2\eta\left(\frac{d_\alpha(x, y)}{d_\alpha(x, z)}\right)\le 2\eta\left(\frac{\diam_{d_\alpha}(T(\xi, n+m))}{\diam_{d_\alpha}(T(\xi, n))/2}\right).
\]
Now note that $\diam_{d_\alpha}(T(\xi, n)) \asymp e^{-n\alpha}$ (where the implicit multiplicative constant does not depend on $\xi$ and $n$). 
  It follows that there exists $C_3>0$ such that
\[\frac{\diam_{d_\alpha}(T(\xi, n+m))}{\diam_{d_\alpha}(T(\xi, n))}\le C_3e^{-m\alpha}\]
for all $\xi\in \lims$ and $n, m\ge 0$. Since $\eta\colon [0, \infty)\arr[0, \infty)$ is a homeomorphism, this implies that there exist $m_0$ and $\lambda_0\in (0, 1)$ such that $\frac{\diam_d(T(\xi, n+m_0))}{\diam_d(T(\xi, n))}\le\lambda_0$ for all $\xi$ and $n$. This implies the statement in \ref{i-diam-exponential}.
\end{proof}

\section{Contraction coefficients}

\subsection{$\ell^p$-contraction} In this section we study $\ell^p$-contraction coefficients, first introduced in the second author's book \cite[\S 4.3.5]{nek:dyngroups}.

Let $(G, \bim)$ be a contracting finitely generated group.
Fix some finite generating set of $G$, and let $l(g)$ be the corresponding word length of $g\in G$.

Fix $p\in (0, \infty)$. The \emph{$\ell^p$-contraction coefficient} of the group is
\[\eta_p=\lim_{n\to\infty}\sqrt[n]{\limsup_{l(g)\to\infty}\frac{\left(\sum_{v\in\alb^n}l(g|_v)^p\right)^{1/p}}{l(g)}}.\]
The $\ell^\infty$-contraction coefficient $\eta_\infty$ is defined in the same way, replacing the $\ell^p$-norm of the vector $(l(g|_v))_{v\in\alb^n}$ by its $\ell^\infty$-norm $\max_{v\in\alb^n}l(g|_v)$. The existence of the first limit in the formula defining $\eta_p$ follows from the lemma below. 

\begin{lem}
\label{lem:subadditive}
Denote 
\[\eta_{p, n}=\limsup_{l(g)\to\infty}\frac{\left(\sum_{v\in\alb^n}l(g|_v)^p\right)^{1/p}}{l(g)}.\]
Then $\eta_{p, n_1+n_2}\le\eta_{p, n_1}\eta_{p, n_2}$.
\end{lem}

\begin{proof}
For every $\lambda_i>\eta_{p, n_i}$, there exists $l_i>0$ such that for every $g$  with $l(g)\ge l_i$ we have
\[\sum_{v\in\alb^{n_i}}l(g|_v)^p\le\lambda_i^p l(g)^p.\]
It follows that there exists a constant $C$ not depending on $g$ (e.g., $C=\max(l_1, l_2)$) such that
\[\sum_{v\in\alb^{n_i}}l(g|_v)^p\le\lambda_i^p (l(g)+C)^p\]
for all $g\in G$. 

Suppose at first that $p\ge 1$. We have then, using the triangle inequality for the $p$-norm:
\begin{multline*}
\sum_{v\in\alb^{n_1+n_2}} l(g|_v)^p=\sum_{v_1\in\alb^{n_1}}\sum_{v_2\in\alb^{n_2}}l(g|_{v_1v_2})^p\le\sum_{v_1\in\alb^{n_1}}\lambda_2^p(l(g|_{v_1})+C)^p\le\\
\left(\lambda_2C|\alb^{n_1}|^{\frac{1}{p}}+\lambda_2\left(\sum_{v_1\in\alb^{n_1}}l(g|_{v_1})^p\right)^{1/p}\right)^p\le\\
\left(\lambda_2C|\alb^{n_1}|^{\frac{1}{p}}+\lambda_2\lambda_1(l(g)+C)\right)^p
\end{multline*}
which implies that
\[\eta_{p, n_1+n_2}\le\lambda_1\lambda_2,\]
for all $\lambda_1>\eta_{p, n_1}$ and $\lambda_2>\eta_{p, n_2}$.

The case $p\in (0, 1)$ is analogous, but using the triangle inequality for the norm $(x_i)\mapsto \sum |x_i|^p$.
\end{proof}

It is shown in~\cite[Proposition~4.3.12]{nek:dyngroups} that $\eta_p$ does not depend on the choice of the finite generating set or the choice of the basis $\alb\subset\bim$. (This is true for all $p\in (0, \infty]$.)

It is shown in~\cite[Proposition~4.3.15]{nek:dyngroups} that for all $0<p<q\le\infty$ we have
\begin{equation}
    \label{eq:etapinequalities}
\eta_p\ge\eta_q\ge|\alb|^{-1},\quad |\alb|^{-1/p}\eta_p\le|\alb|^{-1/q}\eta_q.
\end{equation}
In particular, the function $p\to\eta_p$ is non-increasing and continuous.

\begin{defn}
The \emph{critical exponent} $p_c(G, \bim)$ of a contracting group is the infimum of the set of values $p$ such that $\eta_p<1$.
\end{defn}

\begin{prop}
\label{prop:largescalesumG}
Suppose that $(G, \bim)$ is a finitely generated contracting group, and $l(\cdot)$ is a word length on $G$ (associated to some finite symmetric generating set). Fix  a basis $\alb\subset\bim$ and let $\nuke$ be the corresponding nucleus. Let $p>0$ be such that $\eta_p<1$. Then there exists a constant $C>0$ and $\eta\in (0, 1)$ such that for every $g\in G$ and $n\ge 0$ we have
\[\sum_{v\in\alb^n, g|_v\notin \nuke}l(g|_v)^p\le C\eta^nl(g)^p.\]
In particular
\[\sum_{v\in\alb^\ast, g|_v\notin \nuke}l(g|_v)^p\le C_1l(g)^p,\]
for $C_1=C/(1-\eta)$.
\end{prop}

\begin{proof}
Without loss of generality (since the statement that we want to prove is invariant under changing the word metric to a bi-Lipschitz equivalent one), we may suppose that the generating set $S$ used to define the word metric contains the nucleus and is such that $s|_v\in S$ for every $v\in\xs$ and $s\in S$ (note that every finite generating set of $G$ is contained in a finite generating set satisfying these conditions). Then the  cocycle rule $gh|_v=g|_{h(v)} h|_v$ implies that $l(g|_v)\le l(g)$ for every $g\in G$ and $v\in \alb^\ast$.

We first claim that it is enough to find some $l_0\ge 2$,  and constants $C$ and $\eta$ (possibly depending on $l_0$) such that for every $n\ge 0$, we have
\begin{equation} \label{e-reduction-to-large-l} \sum_{v\in\alb^{n}, l(g|_v)\ge l_0}l(g|_v)^p\le C \eta^n l(g)^p.\end{equation}
Indeed since $G$ is contracting, for every $l_0\ge 2$ there exists $k$ (depending only on $l_0$) such that if $l(g)< l_0$, then $g|_v\in \nuke$ for all $|v|\ge k$.
For $n\ge k$ and for every $v\in \alb^n$, we have that $l(g|_v)$ is not greater than $l(g|_w)$, where $w$ is the unique prefix of $v$ in $\alb^{n-k}$. Since the projection map from $\alb^n$ to $\alb^{n-k}$ is $|\alb|^{k}$-to-1, this implies that
\[\sum_{v\in \alb^n, g|_v\notin \nuke} l(g|_v)^p\le |\alb|^k \cdot  \sum_{v\in \alb^{n-k}, l(g|_v)\ge l_0} l(g|_v)^p. \]
Note also that if $n\le k$, we have $\sum_{v\in \alb^n} l(g|_v)^p\le |\alb|^k l(g)^p$. 
It follows that if we prove \eqref{e-reduction-to-large-l} for some $l_0 \ge 0$ and some constants $C, \eta$ and for every $n \ge 1$, we will also prove the conclusion in the statement  (for some larger constants $C', \eta'$). 

It is also easy to see (arguing similarly) that we can replace $\alb$ by any power $\alb^{n_0}$, i.e., prove \eqref{e-reduction-to-large-l} only for $n$ divisible by $n_0$.
Indeed this follows from  the same reasoning as above,  that implies that for any $n=kn_0+r$, with $0<r<n_0$, we have  
\[\sum_{v\in \alb^n, l(g|_v) \ge l_0}l(g|_v)^p\le |\alb|^{n_0} \sum _{v\in \alb^{kn_0}, l(g|_v)\ge l_0} l(g|_v)^p.\]

If $\eta_p<1$, then there exist $n_0$, $l_0$, and $\eta<1$ such that 
\[\sum_{v\in\alb^{n_0}}l(g|_v)^p\le\eta l(g)^p\]
for all $g\in G$ such that $l(g)\ge l_0$. Redefine the length by setting
\[\overline l(g)=\left\{\begin{array}{rl}l(g) & \text{if $l(g)\ge l_0$}\\ 0 & \text{otherwise.}\end{array}\right.\]
If $\overline{l}(g)=0$,  then $\overline{l}(g|_v)=0$ for all $v\in \alb^{n_0}$, since $l(g|_v)\le l(g)$. If $\overline{l}(g)\neq 0$, then 
\[\sum_{v\in\alb^{n_0}}\overline l(g|_v)^p\le \sum_{v\in\alb^{n_0}} l(g|_v)^p \le \eta l(g)^p= \eta\overline  l(g)^p,\]
so that in either case we have $\sum_{v\in\alb^{n_0}}\overline l(g|_v)^p\le \eta\overline l(g)^p$,
for all $g\in G$. Consequently,
\[\sum_{v\in\alb^{kn_0}, l(g|_v)\ge l_0} l(g|_v)^p=\sum_{v\in\alb^{kn_0}}\overline l(g|_v)^p\le \eta^k \overline l(g)^p\le \eta^k l(g)^p,\]
which finishes the proof. \qedhere
\end{proof}

\begin{defn}
The \emph{portrait} of an element $g\in G$ is the set $P(g)$ of finite words $v\in\xs$ such that $g|_u$ does not belong to the nucleus $\nuke$ of $G$ for any proper prefix $u$ of $v$. If $g\in\nuke$, then we set $P(g)=\{\emptyset\}$.
\end{defn}

The portrait is prefix-closed, i.e., a rooted subtree of $\xs$. Let us denote by $L(g)$ the set of leaves of this subtree, i.e., the set of words $v\in\xs$ such that $g|_v\in\nuke$ but $g|_u\notin\nuke$ for any proper prefix $u$ of $v$.

The following bound on the size of the portrait follows immediatly from Proposition~\ref{prop:largescalesumG}.

\begin{prop}
\label{prop:portraisize}
Let $(G, \bim)$ be a contracting group, and choose a basis $\alb\subset\bim$. Let $p\ge 1$ be such that $\eta_p<1$. Then there exist $\eta\in (0, 1)$ and a constant $C>1$  such that
\[|P(g)\cap\alb^n|\le C\eta ^nl(g)^p\]
and \[|P(g)|\le C_1l(g)^p\]
for every $g\in G$ and $n\ge 0$, with $C_1=C/(1-\eta)$
\end{prop}

The portrait $P(g)$ of an element $g$ together with the labeling of the elements $v\in P(g)$ by the permutations the sections $g|_v$ induce on the first level $\alb\subset\xs$ and the labeling of the elements of $L(g)$ by the sections $g|_v\in\nuke$ uniquely determine $g$. This, together with Proposition~\ref{prop:portraisize}, implies that the growth rate of the group $G$ is bounded from above by $e^{Cn^\alpha}$ for every $\alpha>p_c(G, \bim)$. 

\subsection{Critical exponent and conformal dimension}

\begin{thm}
\label{th:critexp}
Let $(G, \bim)$ be a finitely generated contracting group. Its critical exponent $p_c$ is not greater than $\cdim\lims$.
\end{thm}

\begin{proof}
We can assume that $G$ is infinite (else $p_c=0$).  Hence $\cdim \lims \ge 1$ (see Remark \ref{r-top-dimension}). Let $d$ be an Ahlfors-regular quasiconformal metric on $\lims$. Let $\nu$ be the corresponding Hausdorff measure, and let $\alpha$ be its Hausdorff dimension. Take $\beta>\alpha$. It is enough to prove that $p_c<\beta$, i.e., that $\eta_\beta<1$. 

Let $S$ be a finite symmetric generating set of $G$.
Consider a product $g=s_m\ldots s_2s_1$ of of elements of $S$. Let $n_0$ be such that all sections $s|_v$, for $s\in S$ and $v\in \alb^{n_0}$, belong to the nucleus. Then for every $u\in\alb^n$ and $v_0\in\alb^{n_0}$ consider the sequence
\[v_0u,\quad s_1(v_0u),\quad s_2s_1(v_0u), \ldots,\quad s_m\ldots s_2s_1(v_0u).\]
Denote by $u_0, u_1, \ldots, u_m$ the corresponding suffixes of length $n$. Then $u_i=s_i|_{s_{i-1}\cdots s_1(v_0)}(u_{i-1})$, so the tiles $\Tcal_{u_i}$ and $\Tcal_{u_{i-1}}$ intersect in $\lims$, and the tiles $\Tcal\otimes u_i$ and $\Tcal\otimes u_{i-1}$ intersect in $\limg$.

Choose $\xi_{i, v_0u}\in\Tcal_{u_i}\cap\Tcal_{u_{i-1}}$ for every $i=1, 2, \ldots, m$. Let $R_{i, v_0u}$ and $r_{i, v_0u}$ be the smallest and the largest radii such that 
\[B_d(\xi_{i, v_0u}, r_{i, v_0u})\subseteq T(\xi_{i, v_0u}, n)\subseteq  B_d(\xi_{i, v_0u}, \frac{1}{2}R_{i, v_0u}).\]
By Proposition \ref{pr:confdim}~\ref{i-quasi-round} and by Ahlfors--regularity, we can fix a constant $C>0$ (depending only on the chosen metric and measure on $\lims$)  such that $R_{i, v_0u}/r_{i, v_0u}\in [1, C]$   and such that $\nu(T(\xi_{i, v_0u}, n))\ge C^{-1}r_{i, v_0u}^\alpha\ge C^{-1-\alpha}R_{i, v_0u}^\alpha$.

Denote by $M_n$ the maximum over $\xi\in\lims$ of the minimal radius $R$ such that $T(\xi, n)\subset B_d(\xi, \frac{1}{2}R)$. It follows from Proposition~\ref{pr:confdim} that $M_n\to 0$ as $n\to \infty$.

Consequently (using $\beta>1$),
\begin{multline*}(R_{1, v_0u}+R_{2, v_0u}+\cdots +R_{m, v_0u})^\beta\le m^{\beta-1}(R_{1, v_0u}^\beta+R_{2, v_0u}^\beta+\cdots+R_{m, v_0u}^\beta)\le\\
m^{\beta-1}M_n^{\beta-\alpha}(R_{1, v_0u}^\alpha+R_{2, v_0u}^\alpha+\cdots+R_{m, v_0u}^\alpha)\le\\
m^{\beta-1}M_n^{\beta-\alpha}C^{1+\alpha}(\nu(T(\xi_{1, v_0u}, n))+\nu(T(\xi_{2, v_0u}, n))+\cdots+\nu(T(\xi_{m, v_0u}, n))).
\end{multline*}
Since $\xi_{i, v_0u}\in \mathcal{T}_u$, Lemma \ref{l-multiplicity-shadows} tells us that for every $v_0$ and $i$ we have
\[\sup_{\zeta\in \lims} \sum_{u\in\alb^n} \mathbbm{1}_{T(\xi_{i, v_0u}, n)}(\zeta)\le |\nuke|^3,\]
where we denote by $\mathbbm{1}_{A}$ the indicator function of a subset $A\subset \lims$. It follows that
\[\sum_{v_0\in\alb^{n_0}, u\in\alb^n}\nu(T(\xi_{i, v_0u}, n))=\sum_{v_0\in\alb^{n_0}}\int_{\lims}\left(\sum_{u\in \alb^n} \mathbbm{1}_{T(\xi_{i, v_0u}, n)}\right)d\nu \le |\alb|^{n_0}\cdot |\nuke|^3\cdot\nu(\lims).\]
Hence, adding all the inequalities over all $v_0\in\alb^{n_0}, u\in\alb^n$, we will get
\begin{multline*}\sum_{v_0\in\alb^{n_0}, u\in\alb^n}(R_{1, v_0u}+R_{2, v_0u}+\cdots+R_{m, v_0u})^\beta\le \\ |\alb|^{n_0}\cdot |\nuke|^3\cdot\nu(\lims)\cdot C^{1+\alpha}\cdot M_n^{\beta-\alpha}\cdot m^{\beta -1} \le
KM_n^{\beta-\alpha}m^\beta\end{multline*}
for some constant $K>0$.
\begin{lem}
There exists a constant $C_1>0$, depending only on $G$, $\alb$, and $d$, such that
$l(g|_{v_0u})\le C_1(R_{1, v_0u}+R_{2, v_0u}+\cdots+R_{m, v_0u})+C_1$.
\end{lem}

\begin{proof}
Let us define a metric on $\limg$. Denote $\Tcal S=\bigcup_{s\in S}\Tcal\cdot s$.  Let $\zeta_1, \zeta_2\in\limg$. 
Take all $v_1, v_2\in\bim^{\otimes n}$ for all $n\ge 0$ such that $\zeta_i\in(\Tcal S)\otimes v_i$ and $(\Tcal S)\otimes v_1\cap(\Tcal S)\otimes v_2\ne\emptyset$, and then take the infimum $d_0(\zeta_1, \zeta_2)$ of $\diam((\Tcal S)\otimes v_1/G)+\diam((\Tcal S)\otimes v_2/G)$. Here $\bim^{\otimes 0}=G$. Denote the infimum by $d_0(\zeta_1, \zeta_2)$, which may be infinite. 

We obviously have that $d_0(\zeta_1, \zeta_2)$ is not smaller than the distance between the images of $\zeta_1$ and $\zeta_2$ in $\lims$. 

Define then $d_1(\zeta_1, \zeta_2)$ as the infimum of $d_0(\xi_0, \xi_1)+d_0(\xi_1, \xi_2)+\cdots+d_0(\xi_{n-1}, \xi_n)$ over all sequences $\xi_i$ such that $\xi_0=\zeta_1$ and $\xi_n=\zeta_2$. Then $d_1$ is a finite metric on $\limg$. It is obviously $G$-invariant and, by construction, $(\X, d_1)$ is an almost geodesic space. For every $\xi\in\limg$, the diameter of the union of the tiles of the $n$th level of $\limg$ containing $\xi$ is not larger than the diameter of the union of the tiles of the $n$th level of $\lims$ containing the image of $\xi$. Consequently, the metric $d_1$ is compatible with the topology on $\limg$.

Choose a point $\xi_0\in\Tcal$, and consider the sequence \[\xi_0\otimes v_0u,\quad \xi_0 \cdot s_1\otimes v_0u,\quad \xi_0\cdot s_2s_1 \otimes v_0u, \ldots, \quad\xi_0 \cdot s_m\cdots s_2s_1 \otimes v_0u.\]
Then $R_{1, v_0u}+R_{2, v_0u}+\cdots+R_{m, v_0u}$ is an upper bound on $d_1(\xi_0\otimes v_0u, \xi_0 \cdot g \otimes v_0u)$. The lemma follows then from Lemma~\ref{lem:lengthofsections}.
\end{proof}


Consequently, using the triangle inequality for the $\ell^\beta$-norm (since $\beta>1$),
\begin{multline*}
\left(\sum_{v\in\alb^{n+n_0}}l(g|_v)^\beta\right)^{1/\beta}\le\\
C_1\left(\sum_{v_0\in\alb^{n_0}, u\in\alb^n}(R_{1, v_0u}+R_{2, v_0u}+\cdots+R_{m, v_0u}+1)^\beta\right)^{1/\beta}\le\\
C_1\left(\sum_{v_0\in\alb^{n_0}, u\in\alb^n}(R_{1, v_0u}+R_{2, v_0u}+\cdots+R_{m, v_0u})^\beta\right)^{1/\beta}+C_1|\alb^{n+n_0}|^{1/\beta}\le\\
K_1M_n^{1-\frac{\alpha}{\beta}}\cdot l(g)+C_1|\alb^{n+n_0}|^{1/\beta},
\end{multline*}
for some $K_1>0$, which implies that $\limsup_{l(g)\to\infty}\frac{\left(\sum_{v\in\alb^{n+n_0}}l(g|_v)^\beta\right)^{1/\beta}}{l(g)}\le K_1M_n^{1-\frac{\alpha}{\beta}}$. Since $M_n\to 0$ as $n\to\infty$, Lemma~\ref{lem:subadditive} implies that $\eta_\beta<1$.
\end{proof}

\begin{rem}
The inequality in Theorem~\ref{th:critexp} is not sharp. For example, it is known that $\eta_1$ for the first Grigorchuk group $G$ is strictly less than 1 (this is the classical sum-contraction inequality of Grigorchuk \cite{grigorchukgrowth}), which implies by the inequality~\eqref{eq:etapinequalities} that its critical exponent is strictly less than 1. On the other hand, the limit space of the Grigorchuk group is a segment, hence we have $\cdim\lims=1$. In fact, it follows from the norm contracting inequality in \cite{bar-growth} and the explicit sequence of group elements which satisfies the reverse inequality (see \cite[Prop 4.7]{BE-perm}) that its critical exponent is equal to $\alpha_0=\log 2/\log \lambda_0\simeq 0.7674$, where $\lambda_0$ is the  positive real root of the polynomial $X^3-X^2-2X-4$. Recall the definition of critical constant for Liouville property in the introduction. From the random walk with nontrivial Poisson boundary constructed in \cite{EZ} for the purpose of volume lower estimate, we have that $\rm{Cr}_{Liouv}(G)$ is also equal to $\alpha_0$ and $\alpha_0$ is the growth exponent of $G$. 
\end{rem}




 
 \subsection{Thurston $p$-maps}
 
 Let $(G, \bim)$ be a self-similar group. In this subsection we present another bound on $p_c(G, \bim)$, related to Thurston's theorem (see~\cite{DH:Thurston}) characterizing post-critically finite branched self-coverings of the sphere that are realizable as complex rational functions.  (Note that this result will not be used in the proof of our main theorems in the introduction.) 
 
 Consider the linear span $V$ of the conjugacy classes of infinite order elements of $G$. We denote by $[g]$ the conjugacy class of an element $g\in G$ seen as an element of $V$. If $g$ is of finite order, then we define $[g]$ to be equal to zero. Choose a basis $\alb\subset\bim$ and the associated self-similar action on $\xs$. Let $g\in G$. Consider the action of $g$ on $\alb$, and for every cycle $x_1\mapsto x_2\mapsto \ldots\mapsto x_k\mapsto x_1$ of the action, consider the conjugacy class $[g|_{x_k}\cdots g|_{x_2}g|_{x_1}]=[(g^k)|_{x_1}]$. Note that it does not depend on the choice of the initial element $x_1$ of the cycle. Choose $p\ge 1$, and denote by $T_p([g])$ the sum of the elements $k^{1-p}[(g^k)|_{x_1}]\in V$ taken over all cycles of the action of $g$ on $\alb$. Since a change of the basis $\alb$ corresponds to conjugation of the wreath recursion by an element of the wreath product, the value of $T_p([g])$ does not depend on the choice of $\alb$.
 
 For example, if $a=\sigma(1, a)$ is the binary odometer, then we have $T_p([a])=2^{1-p}[a]$.

  We call the linear operator $T_p$ the \emph{Thurston's $p$-map}. 
 
 \begin{lem}
 \label{lem:Tpiterate}
 If $T_p$ is the Thurston's $p$-map for $(G, \bim)$, then the Thurston's $p$-map for $(G, \bim^{\otimes n})$ is $T_p^n$.
 \end{lem}

 \begin{proof}
Let $x_1\mapsto x_2\mapsto \cdots\mapsto x_k\mapsto x_1$ be a cycle of the action of $g$ on $\alb$. Denote $h=(g^k)|_{x_1}$.  For every cycle $v_1\mapsto v_2\mapsto\cdots\mapsto v_m\mapsto v_1$ of the action of $h$ on $\alb^n$ we get the cycle
\begin{multline*}
x_1v_1\mapsto x_2g|_{x_1}(v_1)\mapsto\cdots\mapsto x_k(g^{k-1})_{x_1}(v_1)\mapsto\\
x_1v_2\mapsto x_2g|_{x_1}(v_2)\mapsto\cdots\mapsto x_k(g^{k-1})_{x_1}(v_2)\mapsto\\
\ldots\\
x_1v_m\mapsto x_2g|_{x_1}(v_m)\mapsto\cdots\mapsto x_k(g^{k-1})_{x_1}(v_m)\mapsto x_1v_1
\end{multline*}
of length $km$ of the action of $g$ on $\alb^{n+1}$. We have 
$(g^{km})|_{x_1v_1}=h^m|_{v_1}$. A proof of the lemma now follows by induction.     
\end{proof}
 
 \begin{thm}
 \label{th:thurstonmap}
 Suppose that $U<V$ is a $T_p$-invariant subspace spanned by a finite set of conjugacy classes. If the $\ell^p$-contraction coefficient $\eta_p$ of $(G, \bim)$ is less than 1, then the spectral radius of $T_p|_U$ is also less than 1.
 \end{thm}
 
 In particular, we can use this theorem to find  lower estimates for $p_c(G, \bim)$, and hence for $\cdim(G, \bim)$.

A particular case of this theorem is related to one direction of Thurston's theorem (see~\cite{DH:Thurston}). In Thurston's theorem one considers subspaces spanned by conjugacy classes defined by disjoint simple closed curves of the punctured sphere (with some additional conditions). Such collections of curves are called \emph{multicurves}.

\begin{proof}
Let $\{[g_1], [g_2], \ldots, [g_m]\}$ be a finite spanning set of $U$. Let $N$ be a natural number such that the action of each element $g_i^N$ on $\alb$ is identical. Denote by $L_{i, kN}$ the length of $g_i^{kN}$ (with respect to some fixed finite generating set of the group). We have then $L_{i, k_1k_2N}\le k_1L_{i, k_2N}$ for all $i=1, \ldots, m$, $k_1, k_2\ge 1$. Equivalently, we have \[L_{i, k_2N}\ge k_1^{-1} L_{i, k_1k_2N}.\]

Note also that $L_{i, kN}\to\infty$ as $kN\to\infty$, since we assume that the elements $g_i$ are of infinite order.

Assume that $\eta_p<1$. Let $\eta$ be such that $\eta_p<\eta<1$. After replacing $\alb$ by $\alb^n$ for a large $n$ (using Lemma~\ref{lem:Tpiterate}) and assuming that $N$ is large enough, we will have
\[\sum_{x\in\alb}l(g_i^{Nk}|_x)^p\le \eta^p l(g_i^{Nk})^p\] for all $i$ and $k$.

For a cycle $x_1\mapsto x_2\mapsto\cdots\mapsto x_r\mapsto x_1$ of $g_i$, the corresponding sections $g_i^{Nk}|_{x_t}$ are conjugate to $\left(g_i^r|_{x_1}\right)^{\frac{Nk}{r}}$. If $g_i^r|_{x_1}$ has infinite order, then it is conjugate to one of the elements $g_j$. Since the number of corresponding conjugators is finite, we get
\[l(g_i^{Nk}|_{x_t})\ge L_{j, Nk/r}-C\]
for some constant $C$ that depends only on the spanning set $\{[g_1], [g_2], \ldots, [g_m]\}$.

Choose an arbitrary $\rho\in (\eta^p, 1)$, so that $\eta^{-p}\rho>1$. Assuming that $N$ is large enough, we will get $l(g_i^{Nk}|_{x_t})\ge \rho^{1/p} L_{j, Nk/r}$ for all $k\ge 1$ and all $1\le t\le r$. Consequently,
\[\sum_{t=1}^{r}l(g_i^{Nk}|_{x_t})^p\ge \rho rL_{j, Nk/r}^p\ge \rho r\left(r^{-1}L_{j, Nk}\right)^p=\rho r^{1-p}L_{j, Nk}^p.\]

It follows that
\[L_{i, Nk}^p=l(g_i^{Nk})^p\ge\eta^{-p}\sum_{x\in\alb}l(g_i^{Nk}|_x)^p\ge \eta^{-p}\rho\cdot \phi(T_p([g_i])),\]
where $\phi$ is the linear functional defined by
\[\phi([g_j])=L_{j, Nk}^p.\]
Consider the matrix of the dual operator $T_p^*$ in the basis of $U^*$ dual to the basis $([g_1], [g_2], \ldots, [g_m])$. The above inequality means that the $i$th coordinate $T_p^*(\phi)([g_i])=\phi(T_p([g_i])$ of $T_p^*(\phi)\in U^*$ is less than or equal to $\eta^p\rho^{-1}$ times the $i$th coordinate $\phi([g_i])$ of $\phi$. Since all coordinates of $\phi$ are positive, the Perron-Frobenius theorem implies that the spectral radius of $T_p^*$ is strictly less than 1. Consequently, the spectral radius of $T_p$ is also strictly less than 1.
\end{proof}

As an example of an application of  Theorem~\ref{th:thurstonmap} (and Thurston's theorem), consider a mating of two cubic polynomials, found by  M.~Shishikura and L.~Tan, see~\cite{mitsutanlei}.

The iterated monodromy group is generated by
\begin{alignat*}{2}
a_1 &= (12)(1, 1, a_3), &\quad b_1&= (1, b_3, 1),\\
a_2 &= (a_1, 1, 1), &\quad b_2&=(b_1, 1, 1),\\
a_3 &= (02)(a_3^{-1}, 1, a_2a_3), &\quad b_3 &= (012)(b_3^{-1}, 1, b_2b_3).
\end{alignat*}
Note that $a_1a_2a_3=b_1b_2b_3=\tau$, where $\tau=(012)(1, 1, \tau)$. The groups $\langle a_1, a_2, a_3\rangle$ and $\langle b_1, b_2, b_3\rangle$ are the iterated monodromy groups of the mated cubic polynomials.

It is checked directly that the group is contracting with the nucleus consisting of the elements $a_3, a_2a_3, b_3, b_2b_3, a_1a_2a_3=b_1b_2b_3$, their inverses, and the identity element. Moreover, the wreath recursion is contracting for the free group (of rank 5), hence the limit space has topological dimension 1, see \cite[Theorem 4.4]{nek:dimension}.



Post-conjugating the wreath recursion by $(1, b_3^{-1}, b_3^{-1})$, and passing to the generating set $b_1, b_2, b_3, x=b_1^{-1}a_1$, $y=b_3a_3^{-1}$ (we have $b_2^{-1}a_2=yx^{-1}$ in the faithful quotient), we get an equivalent wreath recursion:
\begin{alignat*}{2}
b_1 &= (1, b_3, 1), &\quad x&=(12)(1, 1, y^{-1}),\\
b_2 &= (b_1, 1, 1), &\quad y&=(12)(x, 1, y^{-1}),\\
b_3 &= (012)(1, 1, b_2). & & 
\end{alignat*}

Let $U$ be the formal linear span of the conjugacy classes 
$[x], [x^{-1}], [y], [y^{-1}]$.
Thurston's 2-map $T_2$ acts by the rule
\begin{alignat*}{2}
[x]&\mapsto\frac 12[y^{-1}], &\quad[x^{-1}]&\mapsto\frac 12[y],\\
[y]&\mapsto[x]+\frac 12[y^{-1}],&\quad[y^{-1}]&\mapsto[x^{-1}]+\frac 12[y].
\end{alignat*}

It is easy to see that the spectral radius of $T_2|_U$ is equal to 1. In particular, this means, by Thurston's theorem, that the mating is \emph{obstructed}, i.e., is not equivalent to a rational function. Theorem~\ref{th:thurstonmap} implies that the conformal dimension of the limit space of the group is at least 2.

\section{Random walk and Liouville property}
\label{s-Liouville}
\subsection{Random walk preliminaries}
\subsubsection{Notation} Let $G$ be a group endowed with a symmetric probability measure $\mu$. The measure $\mu$  is said to be \emph{non-degenerate} if its support $\supp \mu$ generates $G$ as a semi-group. If $G$ is finitely generated, the measure $\mu$ is said to have \emph{finite $p$-moment} for $p>0$ if $\sum \mu(g) l(g)^p<+\infty$, where $l(\cdot)$ is the word length on $G$ associated to some finite symmetric generating set. We denote by $(\gbf_t)$ the  (left) random walk associated to $\mu$, namely $\gbf_t=\mathbf{s}_t\cdots \mathbf{s}_1$, where $(\mathbf{s}_i)_{i\ge 1}$ is a sequence of independent random variables with distribution $\mu$. 
\subsubsection{Capacities and random walk on Schreier graphs}

 Assume that $G$ acts on a countable set $\Omega$. Then for each starting point $\omega\in \Omega$, the process $(\gbf_t \cdot \omega)$ is a Markov chain on $\Omega$ with transition probabilities $p(v, w)=\sum_{g\cdot v=w} \mu(g)$. The assumption that $\mu$ is symmetric implies that the induced Markov chain is reversible with respect to the counting measure on $\Omega$, which we use as the reference measure in the Dirichlet forms below.
For $p\ge 1$, the  associated \emph{$p$-Dirichlet form} on $\ell^p(\Omega)$ is 
\[\mathcal{E}_{p, \mu}(f)=\frac{1}{2}\sum_{g\in G, v\in\Omega}\mu(g)|f(v)-f(g\cdot v)|^p, \quad \quad f\in \ell^p(\Omega).\]
   Let $A, B$ be two disjoint finite subsets of $\Omega$. For $p\ge 1$, the (effective) \emph{$p$-capacity} between $A$ and $B$ is defined as 
\[\Capc_{p, \mu}(A, B)=\inf\left\{\mathcal{E}_{p, v}(f)\colon  f\in \ell^p(\Omega), f|_A=1, f|_B=0\right\}.\]

  For $p=2$  this  has the following  well-known probabilistic interpretation, which is \cite[Exercise 2.47]{Ly-Pe}. We provide a proof for completeness. 
\begin{lem} \label{l-capacity-rw}
Let $(G, \mu)$ be a group endowed with a probability measure, and $(\gbf_n)$ the associated random walk. Assume that $G$ acts on a finite set $\Omega$. For $\omega\in \Omega$ and $B\subset \Omega$, consider the stopping time:  
\[\Tbf_{\omega\to B}= \min \{t\ge 1 \colon \gbf_t \cdot \omega \in B\}.\]
Then for $A,B$ disjoint finite subsets of $\Omega$,
\[\Capc_{2, \mu}(A, B)=\sum_{\omega \in A} \mathbb{P}(\Tbf_{\omega\to B}<\Tbf_{\omega\to A}).\]

\end{lem}

\begin{proof}
We use the same notation and terminology on electric networks as in \cite[Chapter 2]{Ly-Pe}. (Note that many statements about capacity are proven there assuming that $A=\{a\}$ is a singleton, leaving the generalization to more general subsets as a (usually straightforward) exercise.)
Consider  the induced random walk transition
operator $P_{\mu}$ on $\Omega$  given by $P_\mu(x, y)= \sum_{g\in G}{\bf 1}_{\{y=g\cdot x\}}\mu(g)$. Since $\mu$ is symmetric, the Markov chain determined by $P_\mu$ is reversible and admits the counting measure as a stationary measure, thus can be interpreted as arising from an electric network
with conductance $c(x,y)=P_\mu(x, y)$ (see \cite[\S 2.1]{Ly-Pe}). Denote by $\tau_{\omega\to A}=\min\{t\ge0\colon \gbf_{t}\omega\in A\}$
the hitting time of $A$. Let $v(x)=\mathbb{P}(\tau_{x\to A}<\tau_{x\to B})$.
Then $v$ is a voltage function which is $1$ on the source set $A$,
and $0$ on the sink set $B$. Considering the first step of the random
walk, we have 
 \begin{multline*}
\sum_{\omega\in A}\mathbb{P}(\Tbf_{\omega\to B}<\Tbf_{\omega\to A})=\sum_{\omega\in A}\sum_{x\in\Omega}P_{\mu}(\omega,x)\mathbb{P}(\tau_{x\to B}<\tau_{x\to A})
\\ =\sum_{\omega\in A}\sum_{x\in\Omega}c(\omega,x)(v(\omega)-v(x)).
 \end{multline*}
Regard $i(\omega,x)=c(\omega,x)(v(\omega)-v(x))$ as the current on
the edge $(\omega,x)$. The capacity ${\rm Cap}_{2,\mu}(A,B)$ is
the effective conductance between $A$ and $B$, thus it is equal to the
total amount of current flowing out of $A$ under the voltage function
$v$, which is $\sum_{\omega\in A}\sum_{x\in\Omega}i(\omega,x)$. (This is essentially the definition of capacity between subsets given in \cite{Ly-Pe}, see \cite[Exercise 2.5]{Ly-Pe}, the fact that it coincides with our definition above is a consequence of Thomson's principle \cite[p.35]{Ly-Pe}, see \cite[Exercise 2.13]{Ly-Pe}).\qedhere

\end{proof}

\subsubsection{Liouville property and entropy} 

Let $\mu$ be a probability measure on $G$. A function $f\colon G\to \mathbb{R}$ is \emph{$\mu$-harmonic} if $f(x)=\sum_{g\in G}f(xg)\mu(g)$ for all $x\in G$. By definition the $\mu$-random walk has the \emph{Liouville property} if $\mu$ has trivial Poisson boundary, namely if every bounded $\mu$-harmonic function is constant on the subgroup $\langle \supp\mu\rangle $. 
Furstenberg observed that admitting a non-degenerate measure $\mu$ with the Liouville property implies  amenability of the group $G$ (see \cite[Theorem 4.2]{Kai-Ver}). Conversely, every amenable group admits a symmetric non-degenerate measure with  trivial Poisson boundary,  by a result of Kaimanovich and Vershik \cite[Theorem 4.4]{Kai-Ver} and Rosenblatt \cite{Rosenblatt}.

Recall that given a probability measure $\nu$ supported on a discrete set $\Omega$, the \emph{(Shannon) entropy} of $\nu$ is defined as
\[H(\nu)=-\sum_{\omega\in \Omega} \nu(\omega) \log \nu (\omega),\]
with the convention that $0\log 0=0$. If $\mathbf{X}$ is a discrete random variable taking values in a set $\Omega$, the entropy $H(\mathbf{X})$ is defined as the entropy of the distribution of $\mathbf{X}$. 

Let $\mathbf{X}, \mathbf{Y}$ be discrete random variables, and for any value $y$ taken by $\mathbf{Y}$  with positive probability,  denote by $\mathbf{X}^{(y)}$ the random variable with the conditional distribution of $\mathbf{X}$ given $\{\mathbf{Y}=y\}$. Then the \emph{conditional entropy} of $\mathbf{X}$ given $\mathbf{Y}$ is defined as the expected value of $H(\mathbf{X}^{(y)})$, namely \[H(\mathbf{X}|\mathbf{Y})= \sum_y H(\mathbf{X}^{(y)}) \mathbb{P}(\mathbf{Y}=y).\] 

We resume below some basic properties of Shannon entropy (for proofs, see, e.g., \cite[Chapter 2]{InfoTheory}).
\begin{lem} \label{l-entropy-properties}
\begin{enumerate}[label=(\roman*)]
 
\item \label{i-entropy-bound-1} If $\mathbf{X}$ is a random variable taking values in a finite set $\Omega$, then $H(\mathbf{X})\le \log|\Omega|$.
\item \label{i-entropy-bound-2} The entropy of the joint distribution of two discrete random variables $(\mathbf{X}, \mathbf{Y})$ satisfies 
\[H(\mathbf{X}, \mathbf{Y})=H(\mathbf{X}| \mathbf{Y})+H(\mathbf{Y}).\]

\item \label{i-entropy-bound-3} Assume that a random variable $\mathbf{Y}$ can be expressed as $\mathbf{Y}=f(\mathbf{X}_1,\ldots, \mathbf{X}_k)$ for some discrete random variables $\mathbf{X}_1,\ldots, \mathbf{X}_k$ and some measurable function $f$. Then
\[H(\mathbf{Y})\le  H(\mathbf{X}_1)+\cdots +H(\mathbf{X}_k).\]
 
\end{enumerate}
\end{lem}

We also recall the following well-known fact.
\begin{lem} \label{l-entropy-moment}
Let $G$ be a finitely generated group and $l(\cdot)$ be a length metric on $G$ associated to a finite symmetric generating set $S$ containing the identity. Then for every $G$-valued random variable $\gbf$ we have $H(\gbf)\le (\log|S|+1)\mathbb{E}[l(\gbf)]+1$. 
\end{lem}

\begin{proof} The following computation is identical to the one in the proof \cite[Lemma 12.2]{Kai-hyperbolic-v1}, but since the statement is slightly different we repeat it for completeness. 
Let $p_k=\mathbb{P}(l(\gbf)=k)$, and let $\gbf^{(k)}$ be a random variable with the  distribution of $\gbf$ conditioned to $l(\gbf)=k$. By Lemma \ref{l-entropy-properties}~\ref{i-entropy-bound-2}, we have
 \[H(\gbf)=H(\gbf|l(\gbf))+ H(l(\gbf))= \sum_k p_k H(\gbf^{(k)}) -\sum p_k \log p_k.\]
Since $\gbf^{(k)}\in S^k$, by Lemma \ref{l-entropy-properties}~\ref{i-entropy-bound-1} the first sum is bounded by $\sum_k p_k \log |S|^k=  \log|S| \mathbb{E}[l(\gbf)].$ As for the second sum, using that $x\mapsto -x\log x$ is non-decreasing on $[0, e^{-1}]$, we have
\begin{multline*} \sum_k -p_k \log p_k=\sum_{p_k \ge e^{-k}} -p_k \log p_k +\sum_{p_k \le e^{-k}}-p_k \log p_k \\ \leq \sum_k p_kk+ \sum_k e^{-k}k= \mathbb{E}[l(\gbf)]+\sum_k e^{-k}k,\end{multline*}
and the statement follows since $\sum_k e^{-k}k=\frac{e}{(e-1)^2}<1$. \qedhere

\end{proof}
Assume that $\mu$ is a probability measure of finite entropy on a countable group $G$. Then the limit 
\[h_{\mu}=\lim \frac{1}{n} H(\mu^{*n})\] exists and is called the (Avez) \emph{asymptotic entropy} of $\mu$. 
We recall the fundamental result:

\begin{thm}[Entropy criterion, Kaimanovich--Vershik \cite{Kai-Ver}, Derriennic \cite{Der}] Let $\mu$ be a probability mesaure on a countable group $G$. If $H(\mu)<+\infty$, then  the $\mu$-random walk has the Liouville property if and only if  $h_{\mu}=0$.
\end{thm}

For groups acting on rooted trees, the entropy criterion combined with the previous properties of entropy gives the following sufficient condition for boundary triviality. Some tightly related criteria appeared in   \cite{barthvirag,Kai-munchaussen}.
\begin{prop} \label{p-Munchaussen}
Let $G$ be a finitely generated group of automorphisms of a rooted tree $\alb^*$ endowed with a word length $l(\cdot)$, and $\mu$ be a probability measure on $G$ with finite entropy. Then there is a constant $C>0$ such that for every $n, t>0$  we have
\[H(\mu^{*t}) \le C \mathbb{E}[\sum_{v\in \alb^n} l(\gbf_t|_v)] +C|\alb|^n.\]
In particular, if there exists a sequence $(\varepsilon_n)$ tending to $0$ such that for every $n$ we have 
\[\limsup_{t\to \infty} \frac{1}{t}\mathbb{E}\left[\sum_{v\in \alb^n}l(\gbf_t|_v)\right]\le \varepsilon_n,\]
then $h_\mu=0$.
\end{prop}
(In fact, it is easy to check that the limsup in the statement is a limit, by subadditivity).
\begin{proof}
For every $n$ and $t$, the random variable $\mathbf{g}_t$ is completely determined by the collection of sections $(\mathbf{g}_t|_v)_{v\in X^n}$ and by the permutation $\sigma$ describing the action on level $n$. Thus, by Lemma \ref{l-entropy-properties}~\ref{i-entropy-bound-3},
\[H(\mathbf{g}_t)\le \sum_{v\in X^n} H(\mathbf{g}_t|_v)+H(\sigma).\]
By Lemma \ref{l-entropy-moment}, the first term is bounded by $C_1\mathbb{E}[\sum_{v\in \alb^n} l(\gbf_t|_v)]+|\alb|^n$, where $C_1$ depends only on the generating set defining $l(\cdot)$. The permutation $\sigma$ is completely determined by the first-level permutation of all sections $\mathbf{g}_t|_w$ where $|w|\le n-1$. Thus the total number of possibilities for $\sigma$ is bounded by $\exp(C_2 |\alb|^n)$ (where $C_2$ depends only on $|\alb|$), and $H(\sigma)\le C_2|\alb|^n$ by Lemma \ref{l-entropy-properties}~\ref{i-entropy-bound-1}. The proposition follows. \qedhere

\end{proof}

\subsection{Critical contraction exponent and $p$-capacity}\label{s-energy}

Let $(G, \bim)$ be a finitely generated contracting group with a basis $\alb$, and $\mu$ be a symmetric probability measure on $G$. In this subsection, we will relate the $p$-capacities associated to the action of $G$ on the levels $\alb^n$ with $p_c(G, \bim)$ and $\cdim \lims$.

\begin{prop}
\label{pr:energyestimateforg}
Suppose that $(G, \bim)$ is a contracting finitely generated group. Denote by $l(g)$ the word length with respect to some fixed finite generating set on $G$.

Let $d$ be an Ahlfors-regular metric on $\lims$ of Hausdorff dimension $\alpha$. Let $\beta>\alpha$ and let $p>p_c(G, \bim)$. Pick a  basis $\alb\subset\bim$. 

There exist $C>0$ and $\lambda\in (0, 1)$ such that for every $g\in G$ and every $\ldots x_2x_1\in\xmo$ we have
\[\sum_{v\in\alb^n} d(\ldots x_2x_1g(v),\ldots x_2x_1v)^{\beta}\le C\lambda^nl(g)^p.\]
%
\end{prop}

\begin{proof}
Let $\nu$ be the Hausdorff measure of the metric $d$. 
Take an element $g\in G$. For every $v\in\alb^n$,  if $g|_v\in \nuke$, let $v'\in \xs$ be the shortest prefix of $v$ such that $g|_{v'}\in\nuke$ (note that $v'$ can well be empty or equal to $v$). If such a prefix does not exist, i.e., if $g|_v\notin \nuke$, then define $v'=v$. Let $v''\in\xs$ be such that $v=v'v''$. 

We have $\ldots x_2x_1g(v)\in T(\ldots x_2x_1v, |v''|)$ for every $v\in\alb^n$. It follows that \[d(\ldots x_2x_1g(v), \ldots x_2x_1v)^\alpha\le C_1\nu(T(\ldots x_2x_1v, |v''|)\] for some constant $C_1$. Fix $\varepsilon>0$ to be specified shortly. There exists $C_2>0$ such that we have
\begin{multline}\label{e-distance-sum}\sum_{v\in\alb^n}d(\ldots x_2x_1g(v),\ldots x_2x_1v)^{\beta}\le C_2\sum_{v\in\alb^n}\nu(T(\ldots x_2x_1v, |v''|)^{\frac \beta\alpha}\le\\
C_2\nu(\lims)^{\frac{\beta}{\alpha}}\cdot |\{v\in\alb^n\colon  |v''|\le \lfloor \varepsilon n\rfloor\}|+C_2\sum_{v\in\alb^n, |v''|>\lfloor \varepsilon n \rfloor }\nu(T(\ldots x_2x_1v, |v''|)^{\frac \beta\alpha}.
\end{multline}
Note that if $|v''|\le \lfloor \varepsilon n\rfloor$, then the left-prefix of $v$ of length ${n-\lfloor \varepsilon n\rfloor}$ belongs to the portrait $P(g)$ of $g$, so that
\[
|\{v\in\alb^n\colon  |v''|\le \lfloor \varepsilon n\rfloor\}|\le |\alb^{\lfloor \varepsilon n\rfloor}|\cdot |P(g)\cap \alb^{n-\lfloor \varepsilon n\rfloor}|.\]
By Proposition~\ref{prop:portraisize}, there exists $C_3>0$ and $\eta\in (0, 1)$ such that $|P(g)\cap \alb^{n-\lfloor \varepsilon n\rfloor}|\leq C_3 \eta^{(1-\varepsilon)n}l(g)^p$. Thus, if $\varepsilon$ is small enough so that $\lambda_1\colon =|\alb|^{\varepsilon}\eta^{(1-\varepsilon)}<1$, we obtain 
\begin{equation}
    \label{e-portrait-bound}
|\{v\in\alb^n\colon  |v''|\le \lfloor \varepsilon n\rfloor\}|\le   C_4 \lambda_1^n l(g)^p,\end{equation}
for some $C_4>0$ and $\lambda_1\in (0, 1)$.

To bound the second summand in \eqref{e-distance-sum}, first note that for every $v\in \alb^n$ such that $|v''|>0$, we have $v'\in P(g)$. Moreover, for any $\xi\in \lims$, any $1\le k\le n$ and any fixed of $v_1\in \alb^{n-k}$, Lemma \ref{l-multiplicity-shadows} implies that there are at most $|\nuke|^3$ choices of $v_2\in \alb ^{k}$ such that $\xi\in T(\ldots x_2x_1v_1v_2, k)$.  It follows that for every $\xi\in \lims$ we have
    \[|\{v\in \alb^n \colon |v''|>0, \xi\in T(\ldots x_2x_1v, |v''|)\}|\le |P(g)|\cdot|\nuke|^3.\] 


This implies that 
\begin{multline*}
    \sum_{v\in\alb^n, |v''|>0}\nu(T(\ldots x_2x_1v, |v''|)= \int_{\lims} \left( \sum_{v\in \alb^n, |v''|>0} \mathbbm{1}_{T(\ldots x_2x_1v, |v''|)} \right) d\nu \\ \le|P(g)|\cdot|\nuke|^3 \cdot \nu(\lims) \le C_5\cdot|\nuke|^3\cdot \nu(\lims) l(g)^p,\end{multline*}
where we used again Proposition~\ref{prop:portraisize}, and $C_5=C_3/(1-\eta)$.
Hence there exists $C_6$ such that if $n$ is sufficiently large (so that $\lfloor n\varepsilon \rfloor>0$),
\begin{multline*}
    \sum_{v\in\alb^n, |v''|>\lfloor \varepsilon n \rfloor }\nu(T(\ldots x_2x_1v, |v''|)^{\frac \beta\alpha} \le C_6 \left(\max_{v\in\alb^n, |v''|>\lfloor \varepsilon n\rfloor}\nu(T(\ldots x_2x_1v, |v''|)^{\frac{\beta}{\alpha}-1}\right)l(g)^p \\
    \le C_6 \left(\sup_{\xi\in\lims, k>\lfloor \varepsilon n\rfloor}\nu(T(\xi, k))^{\frac \beta\alpha-1}\right)l(g)^p.
\end{multline*}
Proposition~\ref{pr:confdim} implies that the last quantity is bounded from above by $C_7\lambda_2^n l(g)^p$ for some $C_7>0$ and $\lambda_2\in (0, 1)$. This, combined with \eqref{e-distance-sum} and \eqref{e-portrait-bound}, gives
\[\sum_{v\in\alb^n}d(\ldots x_2x_1g(v),\ldots x_2x_1v)^{\beta}\le C \lambda^n l(g)^p\]
for some $C>0$ and $\lambda=\max(\lambda_1, \lambda_2)\in (0, 1)$, as desired. \qedhere
%
\end{proof}

\begin{thm}
\label{th:energies}
Let $(G, \bim)$ be a contracting finitely generated self-similar group, and let $\mu$ be a symmetric probability measure on $G$ with finite $p$-moment for some $p>p_c(G, \bim)$. 

Let $A$ and $B$ be disjoint closed subsets of $\lims$. Fix $\xi_0=\ldots x_2x_1\in\xmo$, and denote by $A_n$ and $B_n$ the sets of words $v\in\alb^n$ such that the point of $\lims$ represented by the sequence $\ldots x_2x_1v$ belongs to $A$ and $B$, respectively. If $\beta>\cdim\lims$  then 
there exist $C>0$ and $\lambda\in (0, 1)$ such that
\[\Capc_{\beta, \mu}(A_n, B_n)\le C\lambda^n.\]
\end{thm}


\begin{proof}
Let $d$ be an Ahlfors-regular metric on $\lims$ of Hausdorff dimension $\alpha<\beta$.
The algebra of Lipschitz functions $\lims\arr\R$ (with respect to $d$ and to the usual metric on $\R$) is dense in the algebra of all continuous functions, by the Stone-Weierstrass theorem. Therefore,  there exists a Lipschitz function $f_0\colon \lims\arr\R$ such that $|f_0(a)-1|<1/3$ and $|f_0(b)|<1/3$ for all $a\in A$ and $b\in B$. Let $\chi\colon \R\arr\R$ be a Lipschitz function such that $\chi(x)=0$ for all $x\in (-1/3, 1/3)$ and $\chi(x)=1$ for all $x\in (2/3, 5/3)$. Then the function $f=\chi\circ f_0\colon \lims\arr\R$ is a Lipschitz function such that $f(a)=0$ and $f(b)=1$ for all $a\in A$ and $b\in B$. Let $K>0$ be a Lipschitz constant for $f$. Then  Proposition~\ref{pr:energyestimateforg} implies that there exist $C>0$ and $\lambda\in(0, 1)$ 
such that for every $g\in G$
\begin{multline*}  
\sum_{v\in \alb^n} |f(\cdots x_2x_1g(v))-f(\cdots x_2x_1v)|^\beta \\\le K^\beta\sum_{v\in\alb^n} d(\ldots x_2x_1g(v),\ldots x_2x_1v)^{\beta}\le K^\beta C \lambda^n l(g)^p.\end{multline*}
To conclude, note that the function $f$ induces a function $f_n$ on $\alb^n$, given by $f_n(v)= f(\cdots x_2x_1v)$, such that $f_n|_{A_n}\equiv 0$  and $f_n|_{B_n}\equiv 1$. 
The previous computation shows that
\[\Capc_{\beta, \mu}(A_n, B_n)\le \mathcal{E}_{\beta, \mu}(f) \le \frac{1}{2}K^\beta C\left(\sum_{g} l(g)^p \mu(p)\right)\lambda^n,\]
and since the $p$-moment of $\mu$ is finite, this completes the proof. 
\qedhere
\end{proof}




We conclude this subsection by pointing out a consequence of Theorem \ref{th:energies}, which will be used in the proof of Liouville property next subsection. To state it, we first note that if $\cdim \lims<2$ (and $G$ is infinite and finitely generated), then the topological dimension of $\lims$ is 1 (see Remark \ref{r-top-dimension}).  Then  Theorem \ref{th:onedim} applies. After replacing $\bim$ by $\bim^{\otimes n_0}$ if needed, we place ourselves in the setting of Notation \ref{notation:traverses}.

\begin{cor} \label{c-resistance-traverse}
Let $(G, \bim)$ be an infinite finitely generated contracting self-similar group such that $\cdim \lims <2$, and $I\colon \Delta_0\otimes \bim\to \Delta_0$ be a contracting model, where $\Delta_0$ is a locally finite connected graph. Retain Notation \ref{notation:traverses}. Let $\mu$ be a symmetric probability measure on $G$ with finite $p$-moment, where $p>p_c(G, \bim)$. Then  there exists $C>0$ and $\lambda\in (0, 1)$ such that 
\[\Capc_{2, \mu}(U_{z_1, n}, U_{z_2, n}) \le C\lambda^n\]
for any two distinct vertices $z_1$ and $z_2$ of $\Gamma_0$. 
\end{cor}
\begin{proof}

Note first that since $\Delta_0$ is a contracting model, if we replace the basepoint $\xi_0$ in Notation \ref{notation:traverses} by a different basepoint $\xi_0'$, then the corresponding sets $U_{z, n}'$ satisfy $U_{z, n}'=U_{z, n}$ for all $n$ large enough. Therefore, without loss of generality, we can suppose that $\xi_0$ is the  image of a point $\widehat{\xi}_0\in \limg$ under the natural map $\limg\to \Delta_0$ given by Theorem \ref{th:contrmodels}. The map $\limg\to \Delta_0$ also induces in the quotient a map $\lims\to \Gamma_0$. We will  apply Theorem \ref{th:energies} to the sets $A, B\subset \lims$ given by the pre-images of $U_{z_1}, U_{z_2}$. To this end we choose the sequence $\ldots x_2x_1\in \xmo$ in the statement of Theorem \ref{th:energies} so that it represents the same element of $\limg$ as $\widehat{\xi}_0$. With this choice, we have $U_{z_1, n}=A_n$ and $U_{z_2, n}=B_n$. \qedhere
\end{proof}

\subsection{Proof of Liouville property under conformal dimension $<2$}
In this subsection, we complete the proof of the following statement on Liouville property.
\begin{thm} \label{t-Liouville} Let $(G, \bim)$ be finitely generated  contracting group such that $\cdim\lims<2$. Let $p>p_c(G, \bim)$. Then every symmetric probability measure $\mu$ on $G$ with finite $p$-moment has trivial Poisson boundary.
\end{thm}

\begin{cor}
Let $(G, \bim)$ be a finitely generated contracting group. If $\cdim\lims<2$, then $G$ is amenable.
\end{cor}
In the proofs below we take terminology and notation from Subsection \ref{subsec:one-dim}. In particular, we suppose that $I\colon \Delta_0\otimes \bim^{\otimes n_0}\arr\Delta_0$ is a contracting model, where $\Delta_0$ is a locally finite connected graph, and  after replacing $\bim$ by $\bim^{n_0}$ (which does not change the group $G$ as an abstract group) we suppose that $n_0=1$. As in Subsection \ref{subsec:one-dim}, we also assume (upon passing to a barycentric subdivision) that no pair of neighbouring vertices of $\Delta_0$ belong to the same $G$-orbit.  Recall that then $\Gamma_0\colon =\Delta_0/G$ is a finite graph with no loops. For a vertex $z$ of $\Gamma_0$ we define the sets $U_{z,n}$ and the number of traverses $\tau_n$ as in Notation \ref{notation:traverses} and Definition \ref{d-traverses}, and are in position to apply Theorem \ref{th:traverses}. 

For clarity we first prove the theorem under the simplified assumption that $\mu$ is symmetric and finitely supported (which is enough to show amenability of $G$).

\begin{proof}[Proof of Theorem \ref{t-Liouville} (finite support case).]
Let $\mu$ be a symmetric probability measure supported on the finite set  $S$, and $\gbf_t=\sbf_t\cdots \sbf_1$ be the associated random walk.  We wish to estimate the expected number of traverses $\mathbb{E}[\tau_n(\sbf_t\cdots \sbf_1)]$. By Corollary \ref{c-resistance-traverse} we have 
$\max_{z_1, z_2} \Capc_{2, \mu} (U_{z_1, n}, U_{z_2, n})\le C\lambda^n$ for some $C>0$ and $\lambda\in (0, 1)$, where the maximum is taken over all distinct vertices $z_1, z_2$ of $\Gamma_0$. 

Say that a level $n$ traverse from $z_1$ to $z_2$ starts at time $i$ at a vertex $v\in U_{z_1,n}$ if there exists $w\in U_{z_2, n}$ and $j\ge i$ so that $(i, j, v, w)$ is a traverse from $z_1$ to $z_2$. Denote this event by $A(i, v, z_1, z_2)$. Let  $N_{i}$ be the total number of traverses that start at time $i$. Then $N_{i}=\sum_{z_1, z_2} \sum_{v\in U_{z_1, n}} \mathbbm{1}_{A(i, v, z_1, z_2)}$. By the Markov property, the probability $\Pr(A(i, v, z_1, z_2))$ is equal to the probability that the random walk started at $v$ visits $U_{z_2, n}$ before going back to  $U_{z_1, n}$, i.e. to $\mathbb{P}(\Tbf_{v\to U_{z_2, n}}<\Tbf_{v\to U_{z_1, n} })$. Hence it follows from Lemma \ref{l-capacity-rw} that
\[\mathbb{E} [N_i] = \sum_{z_1, z_2} \sum_{v\in U_{z_1, n}} \mathbb{P}(\Tbf_{v\to U_{z_2, n}}<\Tbf_{v\to U_{z_1, n} })\le |\Gamma_0|^2  C\lambda^n,\]
where $|\Gamma_0|$ stands for the size of the vertex set of $\Gamma_0$. Therefore,  we have
\[\mathbb{E}[\tau_n(\sbf_t\cdots \sbf_1)]\le \sum_{i=1}^t \mathbb{E}[N_i]\le |\Gamma_0|^2 C\lambda^nt.\]
By Theorem \ref{th:traverses} we deduce that there is $C_1>0$ (not depending on $n$) such that
\[\limsup_{t\to \infty}\frac{1}{t}\mathbb{E}[\sum_{v\in \alb^n} l(\gbf_t|_v)]\le \liminf_{t\to \infty} \frac{1}{t}\mathbb{E}[\tau_n(\sbf_t\cdots \sbf_1)] \le C_1\lambda^n.\]
Proposition \ref{p-Munchaussen} concludes the proof. \qedhere

\end{proof}

We now move to the general case. 
\begin{proof}[Proof of Theorem \ref{t-Liouville}] 
First note that if $p_c(G, \bim)<1$, then for every $p\in (p_c(G, \bim), 1)$ the volume growth of $G$ is bounded above by $\exp(Cn^p)$ for some $C>0$ (see Proposition \ref{prop:portraisize} and the discussion after it). It follows from a general argument that any measure on $G$ (not necessarily symmetric) of finite $p$-moment, $p>p_c(G, \bim)$, has finite entropy and trivial Poisson boundary, see Corollary 2.3 in \cite{EZ}.

For the rest of the proof we assume that $p_c(G, \bim)\ge 1$. Let $\mu$ be a measure on $G$ of finite $p$-moment, $p>p_c(G, \bim)$. In particular, $\mu$ has finite first moment, and thus expectations of the length of sections $l(\gbf_t|_v)$ are finite. By Corollary \ref{c-resistance-traverse}, we have the capacity estimate
\begin{equation} \max_{z_1, z_2\in \Gamma_{0}} \Capc_{2, \mu} (U_{z_1, n}, U_{z_2, n})\le C\lambda^n. \label{e-resistance-bound} \end{equation}

The main difference from the finite support case is that we can no longer directly apply Theorem \ref{th:traverses} to a word $\sbf_t \cdots \sbf_1$ consisting of samples of $\mu$ because the support of $\mu$ is infinite.

Let us revisit the argument in the proof of Theorem \ref{th:traverses}, using the same notation as there. Fix $\xi_0\in \Delta_0$, and let $\mathcal{T}\subset \Delta_0$ be a compact subset containing $\xi_0$ and such that $I^{|v|}(\xi\otimes v)\in \mathcal{T}$ for every $\xi\in \mathcal{T}$ and $v\in \alb^n$ (see Lemma \ref{lem:lengthofsections}~\ref{i-attractor}).  By the Schwarz--Milnor lemma (and compactness of $\mathcal{T})$, we can also fix a constant $L>0$ such that 
\begin{equation}
\label{e-SM-in-final-theorem} 
L^{-1}l(g)-L\le d(\zeta_1, \zeta_2\cdot g) \le Ll(g)+L
\end{equation}
for every $\zeta_1, \zeta_2\in \mathcal{T}$ and every $g\in G$. 

Let $\sbf_t\cdots \sbf_2\sbf_1$ be a sample of the random walk, and set $\mathbf{g}_i=\sbf_i\cdots \sbf_2\sbf_1$, with $\mathbf{g}_0=1$, and $\xi_i=\xi_0\cdot \mathbf{g}_i$. Let $n$ be large enough and fix $v\in \alb^n$.  Choose a geodesic path $\delta_i$ from $I^n(\xi_{i-1} \otimes v)$ to $I^n(\xi_i \otimes v)$  and let $\gamma=\delta_t\cdots \delta_1$ be their concatenation. By the same arguments as in Theorem~\ref{th:traverses}, the length of $\mathbf{g}|_v$ is bounded by an affine function of the total number of edges crossed by $\gamma$. Consider a crossing of an edge $e$ corresponding to a subcurve $\gamma|_{[t_1, t_2]}$, and let $\delta_i, \delta_j$ be the subcurves containing $\gamma(t_1)$ and $\gamma(t_2)$. If both $\delta_i$ and $\delta_j$ have length less than 1/3, then we can associate to the crossing of $e$ a traverse of the word $\sbf_t\cdots \sbf_2\sbf_1$. The number of edges crossed by $\gamma$ whose beginning or end is inside a path $\delta_i$ of length $\ge 1/3$ is not larger than the length of $\delta_i$ plus $2=6\times ({1}/{3})$, and hence less than 7 times the length of $\delta_i$.  We obtain the estimate
\begin{multline*}
\sum_{v\in \alb^n}l(\mathbf{g}_t|_v)\le C|\alb|^n+ C\tau_n(\mathbf{s}_t\cdots \mathbf{s}_2 \mathbf{s}_1)\\
+C\sum_{i=1}^{t} \, \sum_{\substack{v\in \alb^n\\ d(I^n(\xi_i\otimes v), I^n(\xi_{i-1}\otimes v))\ge 1/3}} d(I^n(\xi_i\otimes v), I^n(\xi_{i-1}\otimes v)),\end{multline*}
for some $C>0$.  By \eqref{e-resistance-bound} and by the same argument as for finitely supported measures, we have  \[\mathbb{E}\left[\tau_n(\mathbf{s}_t\cdots \mathbf{s_2} \mathbf{s}_1)\right]\le C_1\lambda^n\cdot t\]
for some $C_1>0$ and $\lambda\in (0, 1)$.
 To conclude the proof, we need to bound the last sum in the estimate above. 
First observe that 
\[I^n(\xi_i\otimes v)=I^n(\xi_0\otimes \mathbf{s}_i\cdots \mathbf{s}_1(v))\cdot \mathbf{s}_i\cdots \mathbf{s}_1|_v\]
and since the action of $G$ on $\Delta_0$ is isometric and $\mathbf{s}_i\cdots \mathbf{s}_1|_v= (\mathbf{s}_i|_{\mathbf{s}_{i-1} \cdots \mathbf{s}_1(v)}) \cdot (\mathbf{s}_{i-1}\cdots \mathbf{s}_1|_v)$, we have
\[d\left(I^n(\xi_i\otimes v), I^n(\xi_{i-1} \otimes v)\right)=d\left(I^n(\xi_0\otimes\mathbf{s}_i\cdots \mathbf{s}_1(v))\mathbf{s}_i|_{\mathbf{s}_{i-1}\cdots \mathbf{s}_1(v)}, I^n(\xi_0 \otimes \mathbf{s}_{i-1}\cdots \mathbf{s}_1(v))\right).\]
Thus, by \eqref{e-SM-in-final-theorem} we have
\[L^{-1} l(\mathbf{s}_i|_{\mathbf{s}_{i-1}\cdots \mathbf{s}_1(v)})-L \le d\left(I^n(\xi_i\otimes v), I^n(\xi_{i-1} \otimes v)\right) \le L l(\mathbf{s}_i|_{\mathbf{s}_{i-1}\cdots \mathbf{s}_1(v)})+L \]
for every $i$ and $v\in \alb^n$. Note that in particular, if $d\left(I^n(\xi_i\otimes v), I^n(\xi_{i-1} \otimes v)\right)\ge 2L$, then 
\[(2L)^{-1} l(\mathbf{s}_i|_{\mathbf{s}_{i-1}\cdots \mathbf{s}_1(v)}) \le d\left(I^n(\xi_i\otimes v), I^n(\xi_{i-1} \otimes v)\right) \le 2L l(\mathbf{s}_i|_{\mathbf{s}_{i-1}\cdots \mathbf{s}_1(v)}). \]
For $v\in \alb^n$ with $n\ge k$, let us denote by $\pi_{n-k}(v)$ the left prefix of $v$ of length $n-k$. Since the map $I$ is contracting, there exists $C_2>0$ and $\eta\in (0, 1)$ such that for every $n\ge k$ we have
\[d(I^n(\xi_i\otimes v), I^n(\xi_{i-1}\otimes v))\le C_2 \eta^{k} d(I^{n-k}(\xi_i\otimes \pi_{n-k}(v)), I^{n-k}(\xi_{i-1}\otimes \pi_{n-k}(v))).\]
Hence there exists $k>0$ (depending only $C_2, \eta$) such that for every $n\ge k$, every $v\in \alb^n$ and every $i$ such that  $d(I^n(\xi_i\otimes v), I^n(\xi_{i-1}\otimes v))\ge 1/3$, we have 
\[ d(I^{n-k}(\xi_i\otimes \pi_{n-k}(v)), I^{n-k}(\xi_{i-1}\otimes \pi_{n-k}(v))) \ge 4L.\]
The last inequality implies $l(\mathbf{s}_i|_{\mathbf{s}_{i-1}\cdots \mathbf{s}_1(\pi_{n-k}(v))})\ge 2$, since we have \[d(I^{n-k}(\xi_i\otimes \pi_{n-k}(v)), I^{n-k}(\xi_{i-1}\otimes \pi_{n-k}(v)))\le 2L l(\mathbf{s}_i|_{\mathbf{s}_{i-1}\cdots \mathbf{s}_1(\pi_{n-k}(v))}).\] Without loss of generality, we suppose that the generating set used to define $l(\cdot)$ contains the nucleus $\nuke$, so, in particular, the last condition implies that $\mathbf{s}_i|_{\mathbf{s}_{i-1}\cdots \mathbf{s}_1(\pi_{n-k}(v))}\notin \nuke$. We can also suppose that $k$ is large enough so that $d(I^n(\xi_i\otimes v), I^n(\xi_{i-1}\otimes v))\le d((I^{n-k}(\xi_i\otimes \pi_{n-k}(v)), I^{n-k}(\xi_{i-1}\otimes \pi_{n-k}(v))$. Since the map $\pi_{n-k}\colon \alb^n \to \alb^{n-k}$ is $|\alb|^k$-to-1, we deduce that for every $i$ and $n\ge k$ we have
\begin{multline*}
 \sum_{\substack{v\in \alb^n\\ d(I^n(\xi_i\otimes v), I^n(\xi_{i-1}\otimes v))\ge 1/3}} d(I^n(\xi_i\otimes v), I^n(\xi_{i-1}\otimes v))\\ \le |\alb|^k \cdot \sum_{\substack{v\in \alb^{n-k}\\ d(I^{n-k}(\xi_i\otimes v), I^{n-k}(\xi_{i-1}\otimes v))\ge 4L}} d(I^{n-k}(\xi_i\otimes v), I^{n-k}(\xi_{i-1}\otimes v)) \\
 \le 2L |\alb|^k \sum_{{v\in \alb^{n-k},  \mathbf{s}_i |_v \notin \nuke,}} l(\mathbf{s}_i|_v)^p.
\end{multline*}
In the last line we used that $l(\mathbf{s}_i|_v)\le l(\mathbf{s}_i|_v)^p$ since $p>1$ and $l(\mathbf{s}_i|_v)$ is a non-negative integer. By Proposition \ref{prop:largescalesumG}, there exist $C_3>0$ and $\eta_1\in (0, 1)$ such that 
\[\sum_{{v\in \alb^{n} \mathbf{s}_i|_v\notin \nuke}} l(\mathbf{s}_i|_v)^p\le C_3 \eta_1^n l(\mathbf{s}_i)^p.\]
Since $\mu$ has finite $p$-moment, taking the expectation on both sides we find
\[ \mathbb{E}\left( \sum_{{v\in \alb^{n}, \mathbf{s}_i|_v\notin \nuke}} l(\mathbf{s}_i|_v)^p\right)\le  C_4 \eta_1^n, \quad \quad C_4=C_3 \mathbb{E}\left(l(\mathbf{s}_1)^p\right). \]

Therefore there is a constant $C_5$ (depending only on the previous constants $C, C_1, \ldots$) such that for every $n\ge k$ and every $t$ we have
 \begin{multline*}
 \mathbb{E}\left(\sum_{v\in \alb^n}l(\mathbf{g}_t|_v)\right)\le C|\alb|^n+ C\mathbb{E}\left(\tau_n(\mathbf{s}_t\cdots \mathbf{s_2} \mathbf{s}_1)\right) +C\sum_{i=1}^{t} \mathbb{E}\left(\sum_{v\in \alb^{n-k}, \mathbf{s}_i|_v \notin \nuke} l(\mathbf{s}_i|_v)^p\right)\\
 \le C|\alb|^n +C_5(\lambda^n+ \eta_1^{n-k})\cdot t.  \end{multline*}

Hence  Proposition \ref{p-Munchaussen} concludes the proof. \qedhere
\end{proof}

 Theorem \ref{th-intro-ineq} stated in the Introduction follows from Theorems \ref{th:critexp} and \ref{t-Liouville}.

\section{Examples}
\subsection{Automata of polynomial activity growth}\label{subsec:poly-exam}

Let $G$ be a faithful self-similar group acting on $\xs$. Consider, for $g\in G$, the function
\[n\mapsto \alpha_g(n)\colon =|\{v\in\alb^n \colon  g|_v\ne 1\}|\]
counting the number of non-trivial sections on the $n$th level. Since the total set of sections $\{g|_v \colon  v\in\xs\}$ is finite, $\alpha_g(n)$ is the number of paths in a finite directed graph (the \emph{Moore diagram} with the trivial state removed) starting in a vertex. Therefore, the function $\alpha_g(n)$ is either exponentially growing, or is bounded by a polynomial \cite[Theorem A]{sid:cycl}. It is not hard to show that the subset of elements $g\in G$ for which $\alpha_g(n)$ is bounded from above by a polynomial of degree $d$ is a subgroup of $G$. We say that $G$ is \emph{generated by an automaton of polynomial activity growth} if $\alpha_g(n)$ is bounded by a polynomial for every $g\in G$.

Groups generated by automata of polynomial activity growth were introduced by S.~Sidki in~\cite{sid:cycl} (in greater generality than we described here). In particular, he showed that such groups can not have free subgroups.

A special sub-class of such groups are groups generated by \emph{bounded automata}, i.e., consisting of elements $g$ such that $\alpha_g(n)$ is bounded. It includes iterated monodromy groups of all sub-hyperbolic polynomials (see~\cite{bartnek:mand,nek:polynom}). 

L.~Bartholdi and B.~Vir\'ag proved amenability of $\img{z^2-1}$ in~\cite{barthvirag} using random walks. Their methods were extended in~\cite{bkn:amenability} to prove that all groups generated by bounded automata are amenable. Later, this result was extended to groups for which $\alpha_g(n)$ is bounded by polynomials of degree 1 (in~\cite{amirangelvirag:linear}) and 2 (in~\cite{amirangelvirag:quadratic}).

The following theorem is proved in~\cite{nekpilgrimthurston}.

\begin{thm}
\label{th:polactivity}
If $(G, \bim)$ is a self-similar contracting group such that its action on $\xs$ for some basis $\alb\subset\bim$ is a subgroup of the group of automata of polynomial activity growth, then $\cdim\lims\le 1$. \end{thm}

\begin{cor} \label{c-polynomial}
    If $(G, \bim)$ satisfies the conditions of the above theorem, and if  $\mu$ is a symmetric probability measure on $G$ with a finite $p$-moment for  some $p>1$, then the $\mu$-random walk is Liouville. In particular $G$ is amenable.
\end{cor}

Note that groups generated by bounded automata are always contracting, but this is no longer true for groups generated by automata of polynomial activity of degree at least 1. For example, the group generated by the wreath recursion $a=\sigma(1, a)$ and $b=(b, a)$ is generated by an automaton of linear activity growth and is not contracting. As shown by G.~Amir and B.~Vir\'ag in \cite{amirvirag:cubic}, there are automata groups of polynomial activity that do not have the Liouville property for any finitely supported symmetric non-degenerate measure so that the contracting assumption is crucial in Corollary \ref{c-polynomial}.

It follows from the results of~\cite{YNS} that a sufficient condition for a group generated by an automaton of polynomial activity growth to be amenable is recurrence of the orbital graphs of its action on the boundary of the tree $\xs$ (see the proof of \cite[Theorem 5.1]{YNS}). This gives another way to conclude amenability of groups satisfying the conditions of Theorem~\ref{th:polactivity}, see~\cite{nekpilgrimthurston}. In contrast, the statement on the Liouville property  in Corollary \ref{c-polynomial} is  new (even for finitely supported measures when the activity degree is larger than 1). For bounded automata groups, it extends the results of \cite{AAMV} to measures with a moment condition (in that particular case the proof of Theorem \ref{t-Liouville}  can actually be modified to  include measures with finite first moment, by showing that Proposition \ref{prop:portraisize} holds for $p=1$).   

\subsection{Iterated monodromy groups of rational functions}

\begin{thm}
\label{th:quasysimmjulia}
Let $f\in\C(z)$ be a sub-hyperbolic (e.g., post-critically finite) rational function. Then the canonical homeomorphism of the limit space of $\img{f}$ with the Julia set of $f$ is a quasi-symmetry with respect to the visual metric on the limit space and the spherical metric on the Julia set.
\end{thm}

\begin{proof}
This is a result of P.~Ha\"issinksy and K.~Pilgrim. Namely,
it is proved in~\cite[Theorem~6.15]{haispilgr:subhyperbolic} that
the limit dynamical system $\si\colon\lims\arr\lims$ of a contracting self-similar group is \emph{coarse expanding conformal} with respect to a visual metric (see the definition in~\cite{haisinskypilgrim}).

It is also shown in~\cite[Theorem~3.3]{haispilgr:subhyperbolic} that for every sub-hyperbolic rational function $f$ the spherical metric restricted to a neighborhood of the Julia set of $f$ is also coarse expanding conformal. By~\cite[Theorem~2.8.2]{haisinskypilgrim} any topological conjugacy between coarse expanding conformal dynamical systems is a quasi-symmetry.
\end{proof}

\begin{cor}
Let $f\in\C(z)$ be a sub-hyperbolic rational function such that its Julia set is not the whole sphere (equivalently, having at least one attracting cycle). Then $\img{f}$ is amenable. Moreover, if $H\le \img{f}$ is any finitely generated subgroup, and if $\mu$ is any symmetric  probability measure on $H$ with finite second moment, then the random walk generated by $\mu$ is Liouville.
\end{cor}

\begin{proof}
It is enough to prove the second part of the statement, since a group is amenable if and only if all its finitely generated subgroups are. Upon replacing $H$ by a larger finitely generated subgroup, we can assume that it is self-similar  and contains the nucleus of $\img{f}$ (with respect to some basis of the associated biset). Then $H$ is itself a contracting self-similar group with the same limit space as $\img{f}$. By Theorems~\ref{th:quasysimmjulia} and ~\ref{t-Liouville}, and~\ref{th:critexp}, it is enough to show that the Julia set of $f$ has Hausdorff dimension strictly less than 2. This is shown in~\cite[Corollary~6.2]{mcmullen}. 
\end{proof}

Let us consider some examples of iterated monodromy groups that are not generated by automata of polynomial activity growth.

Consider $f(z)=z^2-\frac{1}{16z^2}$. This function is representative of the well-studied family of maps of the form $z^n+\frac{\lambda}{z^m}$, see the survey~\cite{sierpinskijulia}. See the Julia set of this function on Figure~\ref{fig:julia}.

\begin{figure}
    \centering
    \includegraphics[scale=0.85]{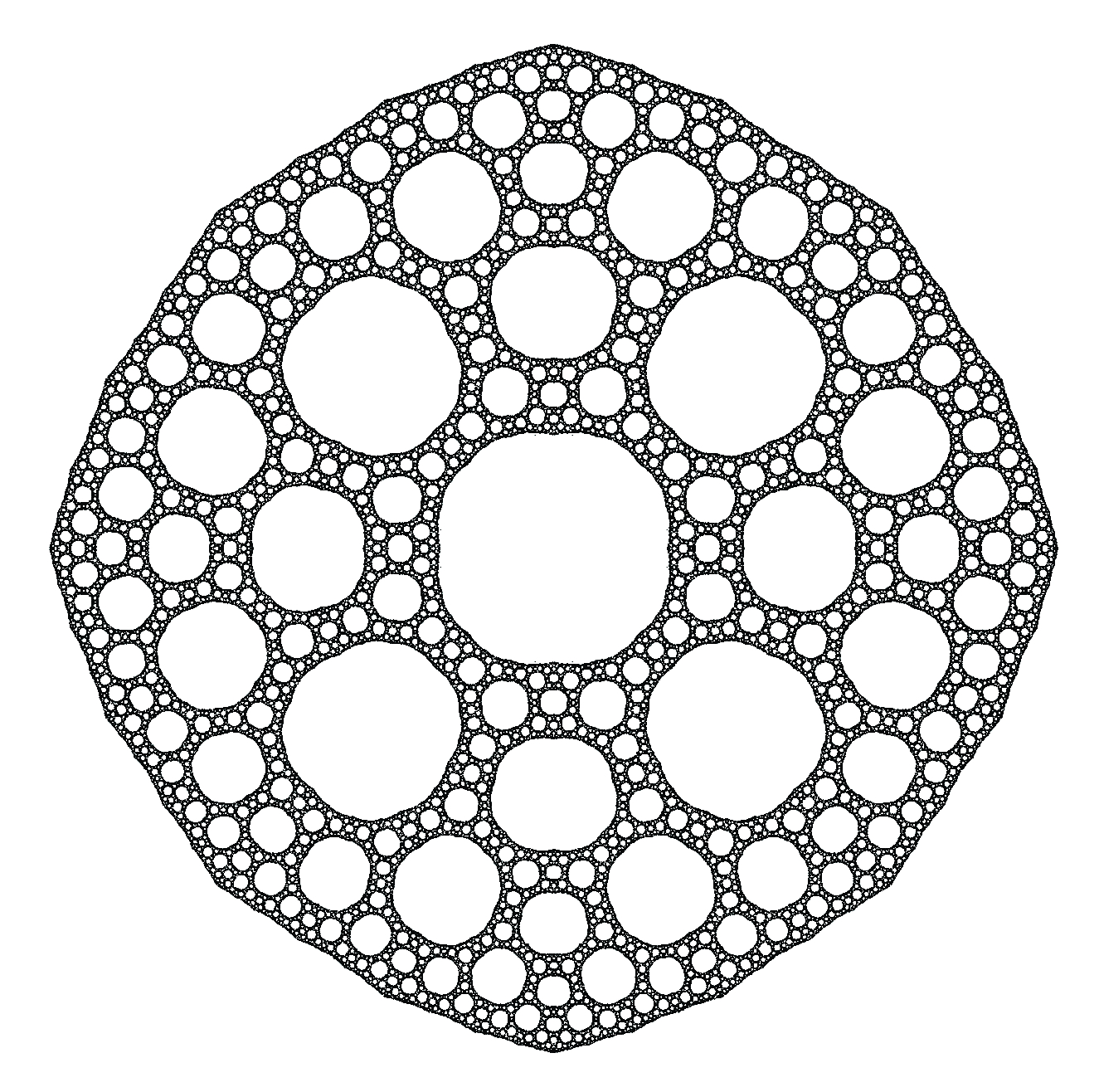}
    \caption{The Julia set of $z^2-\frac{1}{16z^2}$}
    \label{fig:julia}
\end{figure}

Its critical points are $0, \infty$, and roots of $z^4+1/16$. The latter critical points are mapped to $\pm\frac{i}{2}$, which are both mapped to $0$. Consequently, the post-critical set is $P_f=\{0, \infty, i/2, -i/2\}$. Let us take the generators $a, b, c$ of the fundamental group of $\mathbb{PC}^1\setminus P_f$, given by loops around the points $0$, $i/2$, and $-i/2$, respectively, and connected to a basepoint, which we choose to be equal to 1. The preimages of the basepoint 1 are $\pm\sqrt{2\pm\sqrt{5}}/2$ (where the choices of the sign are independent), which are approximately equal to $\pm 0.24i$ and $\pm 1.03$. Let us connect the basepoint to the preimages $\approx\pm 0.24i$ and $\approx 1.03$ by straight paths. We can not do it for $-1.03$, since such a path will contain $0$, instead, let us connect it to the basepoint by a path passing below $0$. We denote the connecting paths from $1$ to $\approx 0.24i$, $\approx -1.03$, $\approx -0.24i$, $\approx 1.03$ by $\ell_1, \ell_2, \ell_3, \ell_4$, respectively. This will be our chosen basis of the biset associated with the rational function (see~\ref{ss:img}). We also denote $\ell_1, \ell_2, \ell_3, \ell_4$ just by $1, 2, 3, 4$ when we write the corresponding wreath recursion.

The generators $a, b, c$, and their preimages $f^{-1}(a), f^{-1}(b), f^{-1}(c)$ together with the connecting paths $\ell_i$ are shown on Figure~\ref{fig:generators}.

\begin{figure}
    \centering
    \includegraphics{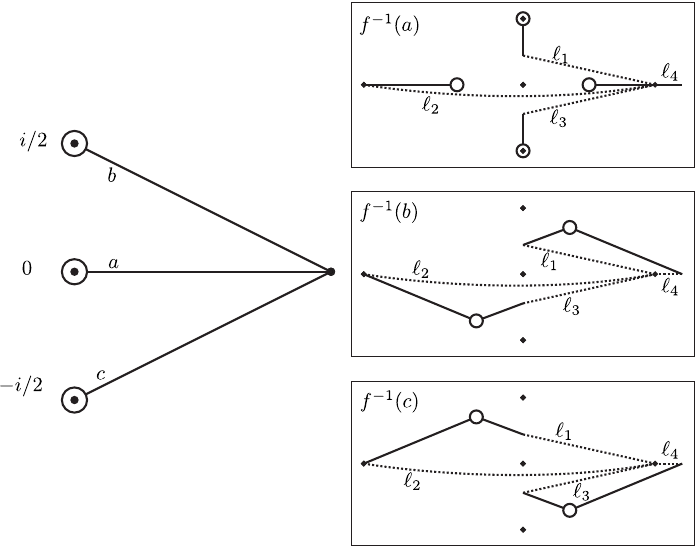}
    \caption{Generators of $\img{z^2-1/16z^2}$}
    \label{fig:generators}
\end{figure}

It follows that the iterated monodromy group is generated by the wreath recursion
\begin{equation} \label{e-recursion-function}
\begin{array}{l} a = (b, 1, c, 1),\\
 b = (23)(14),\\
c = (12)(34)(a, a^{-1}, 1, 1).\end{array}
\end{equation}
Note that all generators are of order 2. (In particular, we may replace $a^{-1}$ by $a$ in the wreath recursion.)










The wreath recursion implies the recurrent rules for the graph of actions on the levels $\alb^n$ of the tree, shown on Figure~\ref{fig:rule}.

\begin{figure}
    \centering
    \includegraphics{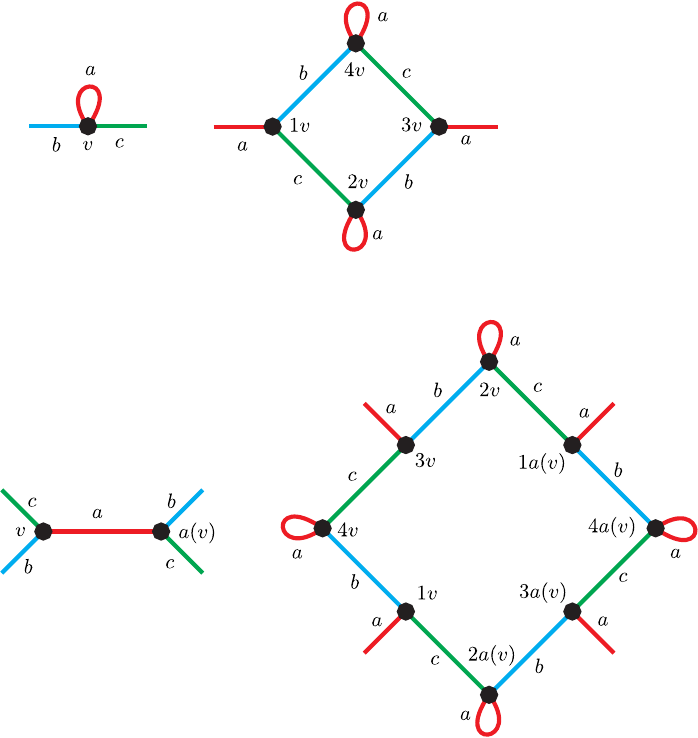}
    \caption{Recursion rules for Schreier graphs  of $\img{z^2-1/16z^2}$. From the wreath recursion \eqref{e-recursion-function}, it is not difficult to show that each Schreier graph can be uniquely decomposed into pieces as in the left part of the picture (vertices fixed by $a$ belong to a top piece, and those moved by $a$ to  a bottom piece). To obtain the $(n+1)$th  Schreier graph from the $n$th one, replace each piece with the corresponding one on the right part of the picture.   }
    \label{fig:rule}
\end{figure}

The same substitution rules from Figure~\ref{fig:rule} are also valid for the infinite orbital Schreier graphs of the action on the boundary $\alb^\omega$ of the tree $\xs$. If we modify the Schreier graphs (of the action of either $\xs$ of $\alb^\omega$)  by imposing  that the edges corresponding to $a$ have length 0 (i.e., contracting them to a single vertex), and that the edges corresponding to $b$ and $c$ have length 1, then we get new graphs that are (uniformly) quasi-isometric to the Schreier graphs (since the group $\langle a\rangle$ is finite). It follows then from the recurrent rules that the shift map $xv\mapsto v$ contracts the distances twice. Since this map is four-to-one, it follows that the orbital Schreier graphs have quadratic growth.

\begin{figure}
    \centering
    \includegraphics[width=\textwidth]{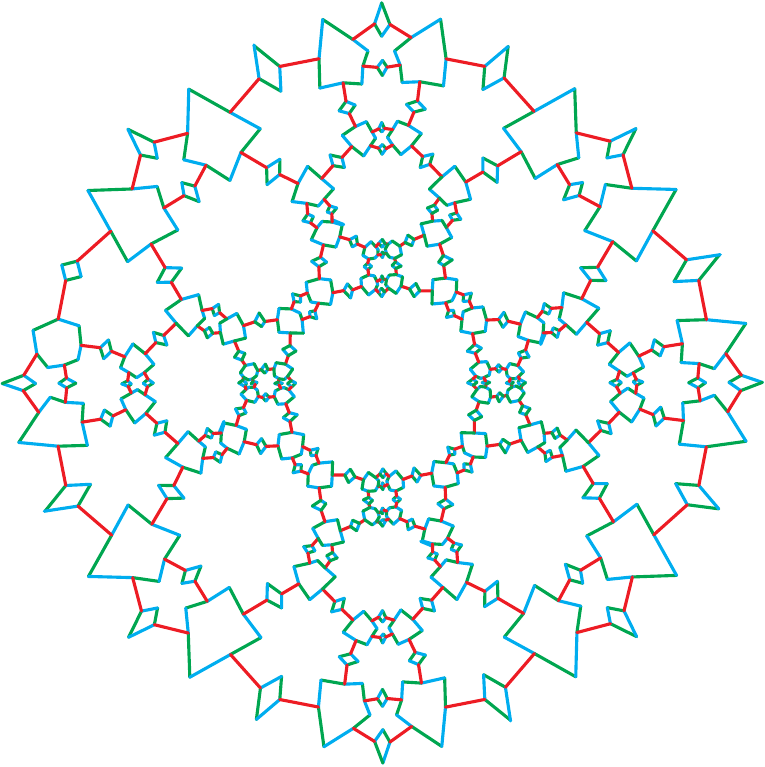}
    \caption{The 5th level Schreier graph of $\img{z^2-1/16z^2}$}
    \label{fig:carpet5a}
\end{figure}

See Figure~\ref{fig:carpet5a} where the graph of the action on the 5th level $\alb^5$ is shown in a way approximating the Julia set of $z^2-\frac{1}{16}z^2$.

\subsection{Sierpi\'nski carpet group}
Let $\alb=\{1,\cdots, 8\}$ be an alphabet. Consider the group $G$ acting on  $\alb^*$  generated by the set $S=\{a, b, c,d\}$, given by the wreath recursion
\begin{alignat*}{1}
a &=(12)(67)(1, 1, a, 1, a, 1, 1, a),\\
b &=(46)(58)(b, b, b, 1, 1, 1, 1, 1),\\
c &=(23)(78)(c, 1, 1, c, 1, c, 1, 1),\\
d &=(14)(35)(1, 1, 1, 1, 1, d, d, d).\\
\end{alignat*}
Its dual Moore diagram is shown in Figure \ref{f-Sierpinski}. All four generators are of order 2.

The limit space $\lims$ can be identified (up to quasisymmetry) with the classical Sierpi\'nski carpet, obtained by successively subdividing a Euclidean square into nine equal squares and removing the middle one. The limit dynamical system $s\colon \lims \to \lims$ can be described by first dilating the carpet by a factor of nine, and then folding it in the natural way, so that each of the eight tiles of the dilated carpet is mapped isometrically onto the original carpet (preserving the orientation  for the four corner tiles, and reversing it for the four remaining ones).

 \begin{figure}
 \centering
     \includegraphics{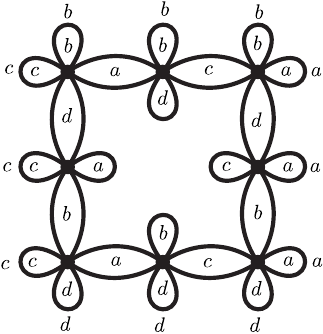}
     \caption{The Sierpi\'nski carpet dual automaton}
     \label{f-Sierpinski}
 \end{figure}
 
 The products $ab, bc, cd$, and $ad$ are of order 6. For example,
 \[ab=(12)(476)(58)(b, b, ba, 1, a, 1, 1, a),\]
 so
 \[(ab)^2=(467)(1, 1, (ba)^2, 1, 1, 1, 1, 1),\]
 and
 \[(ab)^6=(1, 1, (ba)^6, 1, 1, 1, 1, 1).\]
Therefore, $(ab)^6$ acts identically on the first level and on the subtrees $x\xs$ for $x\ne 3$, and as its inverse on the subtree $3\xs$. This shows by induction on the length of words $v\in\xs$ that it acts identically on $\xs$.
 
 We have
 \[ac=(123)(678)(c, a, 1, c, a, c, a, 1),\]
and
 \[(ac)^3=(ac, ca, ac, c, a, ac, ca, ac).\]
This implies that the $ac$-orbit of the word $\underbrace{11\ldots 1}_{\text{$n$ times}}$ has size $3^n$. Consequently, $ac$ has infinite order.
 
Looking at the cycles $(123)$ and $(678)$ of the action of $ac$ on $\alb$ and the corresponding sections, we see that the Thurston's $p$-map $T_p$ satisfies 
\[T_p([ac])=2\cdot 3^{1-p}[ac].\]

It is contracting only when $1-p<\frac{\log 2}{\log 3}$, i.e., for $p>\frac{\log 6}{\log 3}$. Consequently, $ p_c(G, \bim)\ge\frac{\log 6}{\log 3}$, by Theorem~\ref{th:thurstonmap}. In particular, Theorem~\ref{th:critexp} provides the same lower bound for $\cdim(G, \bim)$.
 
 On the other hand, the usual metric on Sierpi\'nski carpet gives the bound $\cdim(G, \bim)\le\frac{\log 8}{\log 3}$. The exact value of $\cdim(G, \bim)$ is still unknown, see some 
 estimates in~\cite{kwapisz}.



\subsection{A non-Liouville contracting group} \label{s-non-Liouville}
Not all contracting self-similar groups are Liouville even for finitely supported generating measures. Let us describe one counter-example.

Consider the group $H$ of affine transformations of the space $\Z_2^3$ (where $\Z_2$ is the ring of dyadic integers) of the form $\vec x\mapsto (-1)^k\vec x+\vec v$, where $\vec v\in\Z^3$ and $k=0, 1$. Note that $H$ has a subgroup of index 2 isomorphic to $\Z^3$ whose action is free,  in particular, it is finitely generated and the orbital Schreier graphs of its action on $\Z_2^3$ are quasi-isometric to $\Z^3$. 

We will denote the points of $\Z_2^3$ as infinite matrices $\left(\begin{array}{ccc}i_{11} & i_{12} & \ldots\\ i_{21} & i_{22} & \ldots\\ i_{31} & i_{32} & \ldots\end{array}\right)$ in such a way that $i_{mn}\in\{0, 1\}$, and if $\left(\begin{array}{ccc} i_{11} & i_{12} & \ldots\\ i_{21} & i_{22} & \ldots\\ i_{31} & i_{32} & \ldots\end{array}\right)$ represents $\vec  x\in\Z_2^3$, and the shifted matrix $\left(\begin{array}{ccc} i_{12} & i_{13} & \ldots\\ i_{22} & i_{23} & \ldots\\ i_{32} & i_{33} & \ldots\end{array}\right)$ represents $\vec y$, then
\[\vec x=-\left(\begin{array}{c}i_1\\ i_2\\ i_3\end{array}\right)+2\vec y.\]
Essentially, we are using the binary numeration system with digits $0$ and $-1$, but write $1$ instead of $-1$.

Denote, for $(i_1, i_2, i_3)\in\{0, 1\}^3$, by $a_{i_1i_2i_3}$ the transformation $\vec x\mapsto -\vec x+\left(\begin{array}{c}i_1\\ i_2\\ i_3\end{array}\right)$. The elements $a_{i_1i_2i_3}$ generate the group $H$. 

We have
\[a_{i_1i_2i_3}\left(-\left(\begin{array}{c}j_1\\ j_2\\ j_3\end{array}\right) +2\vec y\right)=\left(\begin{array}{c}j_1\\ j_2\\ j_3\end{array}\right)-2\vec y+\left(\begin{array}{c}i_1\\ i_2\\ i_3\end{array}\right)=-\left(\begin{array}{c}k_1\\ k_2\\ k_3\end{array}\right)+2\left(-\vec y+\left(\begin{array}{c}l_1\\ l_2\\ l_3\end{array}\right)\right)\]
if $i_t+j_t=2l_t-k_t$. Equivalently, $l_t=\max(i_t, j_t)$ and $k_t$ is $i_t+l_t$ modulo 2.

To shorten notation, let us denote the columns $\left(\begin{array}{c}0\\ 0\\ 0\end{array}\right)$, $\left(\begin{array}{c}0\\ 0\\ 1\end{array}\right)$, $\left(\begin{array}{c}0\\ 1\\ 0\end{array}\right)$,
$\left(\begin{array}{c}0\\ 1\\ 1\end{array}\right)$,
$\left(\begin{array}{c}1\\ 0\\ 0\end{array}\right)$,
$\left(\begin{array}{c}1\\ 0\\ 1\end{array}\right)$,
$\left(\begin{array}{c}1\\ 1\\ 0\end{array}\right)$,
$\left(\begin{array}{c}1\\ 1\\ 1\end{array}\right)$ by
$0, 1, 2, 3, 4, 5, 6, 7$, respectively. Similarly, we will denote the generators $a_{i_1i_2i_3}$ by $a_0, a_1, \ldots, a_7$.
We have then
\begin{align*}
a_0 &=(a_0, a_1, a_2, a_3, a_4, a_5, a_6, a_7),\\
a_1 &=(01)(23)(45)(67)(a_1, a_1, a_3, a_3, a_5, a_5, a_7, a_7),\\
a_2 &=(02)(13)(46)(57)(a_2, a_3, a_2, a_3, a_6, a_7, a_6, a_7),\\
a_3 &=(03)(12)(47)(56)(a_3, a_3, a_3, a_3, a_7, a_7, a_7, a_7),\\
a_4 &=(04)(15)(26)(37)(a_4, a_5, a_6, a_7, a_4, a_5, a_6, a_7),\\
a_5 &=(05)(14)(27)(36)(a_5, a_5, a_7, a_7, a_5, a_5, a_7, a_7),\\
a_6 &=(06)(17)(24)(35)(a_6, a_7, a_6, a_7, a_6, a_7, a_6, a_7),\\
a_7 &=(07)(16)(25)(34)(a_7, a_7, a_7, a_7, a_7, a_7, a_7, a_7).
\end{align*}

Note that $a_0$ is not a section of any of the other generators $a_i$.

The group $H$ acts naturally on the affine space $\R^3$ by the action given by the same formula $\vec x\mapsto (-1)^k\vec x+\vec v$, as for its action on $\Z_2^3$, but with $\vec x\in\R^3$. This action is proper and co-compact. It particular, $H$ can be seen as the fundamental group of the orbifold of the action. The map $\vec x\mapsto 2\vec x$ on $\R^2$ induces a degree 8 self-covering of the orbifold. The fundamental group $H$ with the described action on infinite sequences is the iterated monodromy group of this self-covering. It is a contracting self-similar group, since the self-covering of the orbifold is expanding. For the general theory of iterated monodromy groups of affine orbifolds, see Sections~6.1--6.3 of~\cite{nek:book}. The topological dimension of the orbifold, (which is the limit space of $H$) is 3, hence the conformal dimension is at least 3. Moreover, since the map $\vec x\arr 2\vec x$ is a similarity with respect to the Euclidean metric on $\R^3$, the conformal dimension is not larger than the Hausdorff dimension of $\R^3$; hence, it is equal to 3. This follows, for example, from the results of~\cite{haispilgr:subhyperbolic}.

Let us ``fragment'' $a_0$ into a product of two commuting elements $b$ and $c$ given by
\[b=(b, a_1, 1, 1, 1, 1, 1, 1),\qquad c=(c, 1, a_2, a_3, a_4, a_5, a_6, a_7).\]
Note that $b, c\notin H$ and $bc=a_0$.

Let $G$ be the group generated by $H$ and $b, c$. It follows directly from the formulas that every generator of $G$ acts either as a generator of $H$ or identically. Consequently, the asymptotic equivalence relations on $\xo$ defined by $G$ and $H$ coincide. In particular, the conformal dimension of the limit spaces of $G$ and $H$ are the same.

\begin{prop}
Let $\mu$ be the uniform probability measure on the generating set $\{a_i\}_{i=1, \ldots 7}\cup\{b, c\}$ of $G$. Then the $\mu$-random walk is not Liouville.
\end{prop}

\begin{proof} It follows directly from the recursion that 
if not all letters of $v\in\xs$ are 0 and the first non-zero letter of $v$ is $1$, then $b(v)=a_0(v)$ and $b|_v=a_0|_v$. If not all letters of $v$ are $0$ and the first non-zero letter is not $1$, then $b(v)=v$ and $b|_v=1$. An analogous statement (but the other way around) is true for $c$.

Define the map $\chi\colon G\arr\{0, 1\}$ by the condition
that $\chi(g)=0$ if $g^{-1}|_{\underbrace{00\ldots 0}_{\text{$n$ times}}}\in H$ for all $n$ large enough, and $\chi(g)=1$ otherwise. 

According to the first paragraph of the proof, for every generator $s\in\{a_i\}_{i=1, \ldots, 7}\cup\{b, c\}$ and every $v\in \alb^\ast$, we have $s|_v\in H$ unless $v=\underbrace{00\ldots 0}_{\text{$n$ times}}$ for some $n$ and $s\in \{b, c\}$. It follows that if $g\in G$ and $s\in\{a_i\}_{i=1, \ldots, 7}\cup\{b, c\}$, then $\chi(g)\ne\chi(sg)$ only if $g(00\ldots)=00\ldots$ and $s\in\{b, c\}$.

Consider the $\mu$-random walk on $G$. For every point in the $G$-orbit of $000\ldots$,  one generator in $\{b, c\}$ acts as $a_0$, and the other acts as the identity. It follows that the induced random walk on the $G$-orbit of $000\ldots$ has the same law as a random walk generated by the measure $\nu$ on the generating set $\{a_i\}_{i=0, 1, \ldots, 7}\cup\{1\}$ of $H$, given by $\nu(a_i)=\frac{1}{9}$ for $i\ge 1$, and $\nu(a_0)=\nu(1)=\frac{1}{18}$. Since the Schreier graph is quasi-isometric to $\mathbb{Z}^3$, and simple random walk on $\mathbb{Z}^3$ is transient by P\'olya's theorem, it follows from the comparison theorem (see, e.g., \cite[Theorem 2.17]{Ly-Pe}) that this random walk is transient. Consequently, it will visit $000\ldots$ finitely many times almost surely, and so the value of $\chi$ along a walk trajectory will eventually be constant. The eventual value of $\chi$ will be a non-trivial event invariant under the tail equivalence relation, hence the Poisson boundary of the $\mu$-random walk on $G$ is non-trivial.
\end{proof}

\bibliographystyle{alpha}

\bibliography{biblio.bib}

\begin{bibdiv}
\begin{biblist}

\bib{AAMV}{article}{
      author={Amir, Gideon},
      author={Angel, Omer},
      author={Matte~Bon, Nicol\'{a}s},
      author={Vir\'{a}g, B\'{a}lint},
       title={The {L}iouville property for groups acting on rooted trees},
        date={2016},
     journal={Ann. Inst. Henri Poincar\'{e} Probab. Stat.},
      volume={52},
      number={4},
       pages={1763\ndash 1783},
}

\bib{amirangelvirag:linear}{article}{
      author={Amir, Gideon},
      author={Angel, Omer},
      author={Vir\'ag, B\'alint},
       title={Amenability of linear-activity automaton groups},
        date={2013},
     journal={Journal of the European Mathematical Society},
      volume={15},
      number={3},
       pages={705\ndash 730},
}

\bib{amirangelvirag:quadratic}{unpublished}{
      author={Amir, Gideon},
      author={Angel, Omer},
      author={Vir\'ag, B\'alint},
       title={Amenability of quadratic automaton groups},
        date={2021},
        note={(preprint, arXiv:2111.15206)},
}

\bib{amirvirag:cubic}{article}{
      author={Amir, Gideon},
      author={Vir\'{a}g, B\'{a}lint},
       title={Positive speed for high-degree automaton groups},
        date={2014},
        ISSN={1661-7207},
     journal={Groups Geom. Dyn.},
      volume={8},
      number={1},
       pages={23\ndash 38},
         url={https://doi.org/10.4171/GGD/215},
      review={\MR{3209701}},
}

\bib{bar-growth}{article}{
      author={Bartholdi, Laurent},
       title={The growth of {G}rigorchuk's torsion group},
        date={1998},
        ISSN={1073-7928},
     journal={Internat. Math. Res. Notices},
      number={20},
       pages={1049\ndash 1054},
         url={https://doi.org/10.1155/S1073792898000622},
      review={\MR{1656258}},
}

\bib{BE-perm}{article}{
      author={Bartholdi, Laurent},
      author={Erschler, Anna},
       title={Growth of permutational extensions},
        date={2012},
        ISSN={0020-9910},
     journal={Invent. Math.},
      volume={189},
      number={2},
       pages={431\ndash 455},
         url={https://doi.org/10.1007/s00222-011-0368-x},
      review={\MR{2947548}},
}

\bib{bkn:amenability}{article}{
      author={Bartholdi, Laurent},
      author={Kaimanovich, Vadim},
      author={Nekrashevych, Volodymyr},
       title={On amenability of automata groups},
        date={2010},
     journal={Duke Mathematical Journal},
      volume={154},
      number={3},
       pages={575\ndash 598},
}

\bib{bartnek:mand}{article}{
      author={Bartholdi, Laurent},
      author={Nekrashevych, Volodymyr~V.},
       title={Iterated monodromy groups of quadratic polynomials {I}},
        date={2008},
     journal={Groups, Geometry, and Dynamics},
      volume={2},
      number={3},
       pages={309\ndash 336},
}

\bib{barthvirag}{article}{
      author={Bartholdi, Laurent},
      author={Vir\'ag, B\'alint},
       title={Amenability via random walks},
        date={2005},
     journal={Duke Math. J.},
      volume={130},
      number={1},
       pages={39\ndash 56},
}

\bib{sierpinskijulia}{incollection}{
      author={Blanchard, Paul},
      author={Devaney, Robert~L.},
      author={Look, Daniel~M.},
      author={Moreno~Rocha, Monica},
      author={Seal, Pradipta},
      author={Siegmund, Stefan},
      author={Uminsky, David},
       title={Sierpinski carpets and gaskets as {J}ulia sets of rational maps},
        date={2006},
   booktitle={Dynamics on the {R}iemann sphere},
   publisher={Eur. Math. Soc., Z\"{u}rich},
       pages={97\ndash 119},
}

\bib{carrasco}{article}{
      author={Carrasco~Piaggio, Matias},
       title={On the conformal gauge of a compact metric space},
        date={2013},
        ISSN={0012-9593},
     journal={Ann. Sci. \'{E}c. Norm. Sup\'{e}r. (4)},
      volume={46},
      number={3},
       pages={495\ndash 548 (2013)},
         url={https://doi.org/10.24033/asens.2195},
      review={\MR{3099984}},
}

\bib{Cor-Har-LC}{book}{
      author={Cornulier, Yves},
      author={de~la Harpe, Pierre},
       title={Metric geometry of locally compact groups},
      series={EMS Tracts in Mathematics},
   publisher={European Mathematical Society (EMS), Z\"urich},
        date={2016},
      volume={25},
        ISBN={978-3-03719-166-8},
         url={https://doi.org/10.4171/166},
        note={Winner of the 2016 EMS Monograph Award},
      review={\MR{3561300}},
}

\bib{InfoTheory}{book}{
      author={Cover, Thomas~M.},
      author={Thomas, Joy~A.},
       title={Elements of information theory},
     edition={Second},
   publisher={Wiley-Interscience [John Wiley \& Sons], Hoboken, NJ},
        date={2006},
        ISBN={978-0-471-24195-9; 0-471-24195-4},
      review={\MR{2239987}},
}

\bib{Der}{incollection}{
      author={Derriennic, Yves},
       title={Quelques applications du th\'{e}or\`eme ergodique sous-additif},
        date={1980},
   booktitle={Conference on {R}andom {W}alks ({K}leebach, 1979) ({F}rench)},
      series={Ast\'{e}risque},
      volume={74},
   publisher={Soc. Math. France, Paris},
       pages={183\ndash 201, 4},
}

\bib{DH:Thurston}{article}{
      author={Douady, Adrien},
      author={Hubbard, John~H.},
       title={A proof of {Thurston's} topological characterization of rational
  functions},
        date={1993},
     journal={Acta Math.},
      volume={171},
      number={2},
       pages={263\ndash 297},
}

\bib{erschler-critical}{article}{
      author={Erschler, Anna},
       title={Critical constants for recurrence of random walks on
  {$G$}-spaces},
        date={2005},
        ISSN={0373-0956},
     journal={Ann. Inst. Fourier (Grenoble)},
      volume={55},
      number={2},
       pages={493\ndash 509},
         url={http://aif.cedram.org/item?id=AIF_2005__55_2_493_0},
      review={\MR{2147898}},
}

\bib{EZ}{article}{
      author={Erschler, Anna},
      author={Zheng, Tianyi},
       title={Growth of periodic {G}rigorchuk groups},
        date={2020},
        ISSN={0020-9910},
     journal={Invent. Math.},
      volume={219},
      number={3},
       pages={1069\ndash 1155},
}

\bib{ghysdelaharpe}{incollection}{
      author={Ghys, \'{E}tienne},
      author={de~la Harpe, Pierre},
       title={Le bord d'un espace hyperbolique},
        date={1990},
   booktitle={Sur les groupes hyperboliques d'apr\`es {M}ikhael {G}romov
  ({B}ern, 1988)},
      series={Progr. Math.},
      volume={83},
   publisher={Birkh\"{a}user Boston, Boston, MA},
       pages={117\ndash 134},
         url={https://doi.org/10.1007/978-1-4684-9167-8_7},
      review={\MR{1086655}},
}

\bib{grigorchukgrowth}{article}{
      author={Grigorchuk, R.~I.},
       title={Degrees of growth of finitely generated groups and the theory of
  invariant means},
        date={1984},
        ISSN={0373-2436},
     journal={Izv. Akad. Nauk SSSR Ser. Mat.},
      volume={48},
      number={5},
       pages={939\ndash 985},
      review={\MR{764305}},
}

\bib{zukgrigorchuk:3st}{article}{
      author={Grigorchuk, Rostislav~I.},
      author={{\.Z}uk, Andrzej},
       title={On a torsion-free weakly branch group defined by a three state
  automaton},
        date={2002},
     journal={Internat. J. Algebra Comput.},
      volume={12},
      number={1},
       pages={223\ndash 246},
}

\bib{haisinskypilgrim}{article}{
      author={Ha\"{\i}ssinsky, Peter},
      author={Pilgrim, Kevin~M.},
       title={Coarse expanding conformal dynamics},
        date={2009},
     journal={Ast\'erisque},
      number={325},
       pages={viii+139 pp.},
}

\bib{haispilgr:subhyperbolic}{article}{
      author={Ha\"{\i}ssinsky, Peter},
      author={Pilgrim, Kevin~M.},
       title={Finite type coarse expanding conformal dynamics},
        date={2011},
     journal={Groups Geom. Dyn.},
      volume={5},
      number={3},
       pages={603\ndash 661},
}

\bib{haispilgr:inequiv}{article}{
      author={Ha\"{\i}ssinsky, Peter},
      author={Pilgrim, Kevin~M.},
       title={Quasisymmetrically inequivalent hyperbolic {J}ulia sets},
        date={2012},
     journal={Rev. Mat. Iberoam.},
      volume={28},
      number={4},
       pages={1025\ndash 1034},
}

\bib{Heinonen}{book}{
      author={Heinonen, Juha},
       title={Lectures on analysis on metric spaces},
      series={Universitext},
   publisher={Springer-Verlag, New York},
        date={2001},
}

\bib{JMMS:extensive}{article}{
      author={Juschenko, Kate},
      author={{Matte~Bon}, Nicol\'as},
      author={Monod, Nicolas},
      author={de~la Salle, Mikael},
       title={Extensive amenability and an application to interval exchanges},
        date={2018},
     journal={Ergodic Theory and Dynamical Systems},
      volume={38},
      number={1},
       pages={195\ndash 219},
}

\bib{juschenkomonod}{article}{
      author={Juschenko, Kate},
      author={Monod, Nicolas},
       title={Cantor systems, piecewise translations and simple amenable
  groups},
        date={2013},
     journal={Ann. of Math. (2)},
      volume={178},
      number={2},
       pages={775\ndash 787},
}

\bib{YNS}{article}{
      author={Juschenko, Kate},
      author={Nekrashevych, Volodymyr},
      author={de~la Salle, Mikael},
       title={Extensions of amenable groups by recurrent groupoids},
        date={2016},
     journal={Invent. Math.},
      volume={206},
      number={3},
       pages={837\ndash 867},
}

\bib{Kai-hyperbolic-v1}{article}{
      author={Kaimanovich, Vadim~A.},
       title={The {P}oisson formula for groups with hyperbolic properties},
        date={2000},
        ISSN={0003-486X,1939-8980},
     journal={Ann. of Math. (2)},
      volume={152},
      number={3},
       pages={659\ndash 692},
         url={https://doi.org/10.2307/2661351},
        note={{R}eference to the extended preprint version:
  arXiv:math/9802132v1.},
}

\bib{Kai-munchaussen}{article}{
      author={Kaimanovich, Vadim~A.},
       title={``{M}\"{u}nchhausen trick'' and amenability of self-similar
  groups},
        date={2005},
        ISSN={0218-1967},
     journal={Internat. J. Algebra Comput.},
      volume={15},
      number={5-6},
       pages={907\ndash 937},
}

\bib{Kai-Ver}{article}{
      author={Ka\u{\i}manovich, V.~A.},
      author={Vershik, A.~M.},
       title={Random walks on discrete groups: boundary and entropy},
        date={1983},
        ISSN={0091-1798},
     journal={Ann. Probab.},
      volume={11},
      number={3},
       pages={457\ndash 490},
}

\bib{kigamibook}{book}{
      author={Kigami, Jun},
       title={Geometry and analysis of metric spaces via weighted partitions},
      series={Lecture Notes in Mathematics},
   publisher={Springer, Cham},
        date={[2020] \copyright 2020},
      volume={2265},
        ISBN={978-3-030-54154-5; 978-3-030-54153-8},
         url={https://doi.org/10.1007/978-3-030-54154-5},
      review={\MR{4175733}},
}

\bib{kwapisz}{article}{
      author={Kwapisz, Jaroslaw},
       title={Conformal dimension via {\tt p}-resistance: {S}ierpi\'{n}ski
  carpet},
        date={2020},
     journal={Ann. Acad. Sci. Fenn. Math.},
      volume={45},
      number={1},
       pages={3\ndash 51},
}

\bib{Ly-Pe}{book}{
      author={Lyons, Russell},
      author={Peres, Yuval},
       title={Probability on trees and networks},
      series={Cambridge Series in Statistical and Probabilistic Mathematics},
   publisher={Cambridge University Press, New York},
        date={2016},
      volume={42},
        ISBN={978-1-107-16015-6},
}

\bib{mackay_tyson}{book}{
      author={Mackay, John~M.},
      author={Tyson, Jeremy~T.},
       title={Conformal dimension},
      series={University Lecture Series},
   publisher={American Mathematical Society, Providence, RI},
        date={2010},
      volume={54},
        note={Theory and application},
}

\bib{mcmullen}{article}{
      author={McMullen, Curtis~T.},
       title={Hausdorff dimension and conformal dynamics. {I}{I}.
  {G}eometrically finite rational maps},
        date={2000},
     journal={Comment. Math. Helv.},
      volume={75},
      number={4},
       pages={535\ndash 593},
}

\bib{milnor}{book}{
      author={Milnor, John},
       title={Dynamics in one complex variable},
     edition={Third},
      series={Annals of Mathematics Studies},
   publisher={Princeton University Press},
     address={Princeton, NJ},
        date={2006},
      volume={160},
}

\bib{nek:hyplim}{article}{
      author={Nekrashevych, Volodymyr},
       title={Hyperbolic spaces from self-similar group actions},
        date={2003},
     journal={Algebra and Discrete Mathematics},
      volume={1},
      number={2},
       pages={68\ndash 77},
}

\bib{nek:book}{book}{
      author={Nekrashevych, Volodymyr},
       title={Self-similar groups},
      series={Mathematical Surveys and Monographs},
   publisher={Amer. Math. Soc., Providence, RI},
        date={2005},
      volume={117},
}

\bib{nek:cpalg}{article}{
      author={Nekrashevych, Volodymyr},
       title={{$C^*$}-algebras and self-similar groups},
        date={2009},
     journal={Journal {f\"ur} die reine und angewandte Mathematik},
      volume={630},
       pages={59\ndash 123},
}

\bib{nek:polynom}{incollection}{
      author={Nekrashevych, Volodymyr},
       title={Combinatorics of polynomial iterations},
        date={2009},
   booktitle={Complex dynamics -- families and friends},
      editor={Schleicher, D.},
   publisher={A K Peters},
       pages={169\ndash 214},
}

\bib{nek:models}{article}{
      author={Nekrashevych, Volodymyr},
       title={Combinatorial models of expanding dynamical systems},
        date={2014},
     journal={Ergodic Theory and Dynamical Systems},
      volume={34},
       pages={938\ndash 985},
}

\bib{nek:dyngroups}{book}{
      author={Nekrashevych, Volodymyr},
       title={Groups and topological dynamics},
      series={Graduate Studies in Mathematics},
   publisher={American Mathematical Society, Providence, USA},
        date={2022},
      volume={223},
}

\bib{nek:dimension}{unpublished}{
      author={Nekrashevych, Volodymyr},
       title={Topological dimension and self-similar groups},
        date={2023},
        note={(preprint arXiv:2304.11232)},
}

\bib{nekpilgrimthurston}{unpublished}{
      author={Nekrashevych, Volodymyr},
      author={Pilgrim, Kevin},
      author={Thurston, Dylan},
       title={On amenability of iterated monodromy groups},
        date={2020},
        note={(preprint)},
}

\bib{pansu}{article}{
      author={Pansu, Pierre},
       title={Dimension conforme et sph\`ere \`a l'infini des vari\'{e}t\'{e}s
  \`a courbure n\'{e}gative},
        date={1989},
     journal={Ann. Acad. Sci. Fenn. Ser. A I Math.},
      volume={14},
      number={2},
       pages={177\ndash 212},
}

\bib{PZ20}{article}{
      author={Peres, Yuval},
      author={Zheng, Tianyi},
       title={On groups, slow heat kernel decay yields liouville property and
  sharp entropy bounds},
        date={2020},
     journal={International Mathematics Research Notices},
      volume={3},
       pages={722\ndash 750},
}

\bib{Rosenblatt}{article}{
      author={Rosenblatt, Joseph},
       title={Ergodic and mixing random walks on locally compact groups},
        date={1981},
        ISSN={0025-5831},
     journal={Math. Ann.},
      volume={257},
      number={1},
       pages={31\ndash 42},
}

\bib{mitsutanlei}{article}{
      author={Shishikura, Mitsuhiro},
      author={Lei, Tan},
       title={A family of cubic rational maps and matings of cubic
  polynomials},
        date={2000},
     journal={Experiment. Math.},
      volume={9},
      number={1},
       pages={29\ndash 53},
}

\bib{sid:cycl}{article}{
      author={Sidki, Said~N.},
       title={Automorphisms of one-rooted trees: growth, circuit structure and
  acyclicity},
        date={2000},
     journal={J. of Mathematical Sciences (New York)},
      volume={100},
      number={1},
       pages={1925\ndash 1943},
}

\end{biblist}
\end{bibdiv}

\bigskip

{\small

\noindent\textit{Nicol\'as Matte Bon\\
	CNRS \&
	Institut Camille Jordan (ICJ, UMR CNRS 5208)\\
	Universit\'e Lyon 1\\
	43 blvd.\ du 11 novembre 1918,	69622 Villeurbanne,	France\\}
\href{mailto:mattebon@math.univ-lyon1.fr}{mattebon@math.univ-lyon1.fr}

\smallskip

\noindent\textit{Volodymyr Nekrashevych\\
	Department of Mathematics\\
 Texas A\& M University, College Station, TX 77843-3368, United States of America\\}
\href{mailto:nekrash@math.tamu.edu}{nekrash@math.tamu.edu}

\smallskip

\noindent\textit{Tianyi Zheng\\
	Department of Mathematics\\
 University of California, San Diego (UCSD),
9500 Gilman Drive \# 0112,
La Jolla, CA  92093-0112, United States of America\\}
\href{mailto:tzheng2@math.ucsd.edu}{tzheng2@math.ucsd.edu}

}

\end{document}